\documentclass[12pt]{amsart}

\usepackage{graphicx}
\usepackage{amssymb, comment, color}
\usepackage[toc,page]{appendix}
\usepackage{epsf,overpic,amssymb,amscd,times,euscript}

\usepackage[utf8]{inputenc}
\usepackage[english]{babel}
\usepackage{amsmath}
\usepackage{amsthm}
\usepackage{amsfonts}
\usepackage{amssymb}
\usepackage{amscd}
\usepackage{bezier}
\usepackage{enumerate}

\usepackage{graphicx}
\usepackage{hyperref}
\usepackage{upgreek}

\usepackage{parskip}

\usepackage{bbm}

\newcommand{\id}{id}

\newcommand{\Z}{\mathbbm{Z}}                     % the integer numbers
\newcommand{\R}{\mathbbm{R}}                     % the real line
\newcommand{\C}{\mathbbm{C}}                     % the complex plane
\newcommand{\T}{\mathcal{T}}                     % the torus
\newcommand{\J}{\mathcal{J}}
\newcommand{\D}{\mathbb{D}}
\newcommand{\M}{\mathcal{M}}
\renewcommand{\P}{\mathcal{P}}
\newcommand{\cz}{{\rm CZ}}
\renewcommand{\Im}{\mathrm{Im}}

\newcommand{\CH}{\mathrm{CH}} %cylindrical contact homology
\newcommand{\CC}{\mathrm{CC}} %cylindrical contact complex
\newcommand{\util}{\widetilde{u}}

\newcommand{\wtil}{\widetilde{w}}
\newcommand{\otheta}{\mathbf{\theta}}
\newcommand{\ovr}{\mathbf{r}}
\newcommand{\ot}{\mathbf{t}}

\newtheorem{thm}{Theorem}[section]               % numbered absolutely
\newtheorem*{thm*}{Theorem}               % no number
\newtheorem*{cor*}{Corollary}        % no number
\newtheorem{lem}[thm]{Lemma}  
\newtheorem*{lem*}{Lemma}
\newtheorem{prop}[thm]{Proposition}     % numbered along with Theorem
\theoremstyle{definition}
\newtheorem{defn}[thm]{Definition}      % numbered along with Theorem
\newtheorem{rem}[thm]{Remark}           % numbered along with Theorem  
 \newtheorem*{acknowledgement*}{\protect\acknowledgementname}
\newcounter{claim}

 \providecommand{\acknowledgementname}{Acknowledgement}

\author{Marcelo R.R. Alves}
\thanks{M.R.R. Alves was supported by the ERC consolidator grant 646649  ``SymplecticEinstein'' and by the SFB/TRR 191 ``Symplectic Structures in Geometry, Algebra and Dynamics''.}
\address{Marcelo R.R. Alves, D\'epartement de Math\'ematique\\
Universit\'e Libre de Bruxelles, CP 218,
Boulevard du Triomphe,
B-1050 Bruxelles,
Belgium.}
\email{\texttt{marcelorralves@gmail.com}}

\author{Abror Pirnapasov}
\thanks{A. Pirnapasov was supported by the SFB/TRR 191 ``Symplectic Structures in Geometry, Algebra and Dynamics''.}
\address{Abror Pirnapasov, Fakult\"at f\"ur Mathematik \\
Ruhr-Universit\"at Bochum \linebreak Lehrstuhl X (Analysis), Fach 55
Geb\"aude IB, Etage 3, Raum 59
D-44780 Bochum, Germany.}
\email{\texttt{Abror.Pirnapasov@rub.de }}

\title[Reeb orbits that force $h_{\rm top}$]{Reeb orbits that force topological entropy}
\begin{document}

\begin{abstract}
We develop a forcing theory of topological entropy for Reeb flows in dimension $3$. A transverse link $L$ in a closed contact $3$-manifold $(Y,\xi)$ is said to force topological entropy if $(Y,\xi)$ admits a Reeb flow with vanishing topological entropy, and every Reeb flow on $(Y,\xi)$ realizing $L$ as a set of periodic Reeb orbits has positive topological entropy. Our main results establish topological conditions on a transverse link $L$  which imply that $L$ forces topological entropy. These conditions are formulated in terms of two Floer theoretical invariants: the cylindrical contact homology on the complement of transverse links introduced by Momin in \cite{Momin}, and the strip Legendrian contact homology on the complement of transverse links, introduced in \cite{thesis} and further developed here. We then use these results to show that on every closed contact $3$-manifold that admits a Reeb flow with vanishing topological entropy, there exists transverse knots that force topological entropy.
\end{abstract}

\maketitle
\tableofcontents

\section{Introduction and main results}

It is well known that certain periodic motions can force the existence of a complicated orbit structure for a dynamical system. For instance, a periodic orbit with prime period $3$ of a continuous map of the interval forces the dynamics to have a chaotic behavior as shown in the celebrated paper by Li and Yorke~\cite{LY}.

In surface dynamics,  Nielsen-Thurston theory can be used to obtain forcing results for dynamical complexity.   More precisely, a periodic orbit $\mathcal{O}$ of an orientation preserving homeomorphism $f$ of a closed oriented surface forces infinitely many periodic orbits and positivity of topological entropy if $f$ is isotopic to a pseudo-Anosov homeomorphism relative to $\mathcal{O}$ through an isotopy that fixes $\mathcal{O}.$ A partial order organizes these forced periodic orbits as in the pioneering work of Sharkovskii. See the survey~\cite{Boyland} for a nice exposition of this theory.

Here we study dynamical complexity of Reeb flows, focusing on the $3$-dimensional case. The measure of complexity that we study in this paper is the topological entropy, which codifies in a single non-negative number the exponential complexity of a dynamical system; see section \ref{section:background} for a precise definition. 

Recall that a $1$-form on a $2n-1$-dimensional manifold $Y$ is called a contact form if $\lambda\wedge (d\lambda)^{n-1}$.
The resulting co-oriented hyperplane field $\xi := \ker \lambda$ is then called a co-oriented contact structure and the pair $(Y,\xi)$ is called a co-oriented contact manifold. If $\lambda$ and $\lambda'$ are contact forms on $(Y,\xi)$ then there exists a positive function $f: Y \to (0,+\infty)$ such that $\lambda'=f\lambda$. Given a contact form $\lambda$ on $(Y,\xi)$, its Reeb vector field $X_\lambda$ is the unique vector field  defined by $$\lambda(X_\lambda)\equiv 1 \qquad \mbox{ and } \qquad d\lambda (X_\lambda,\cdot) \equiv 0.$$
The set of Reeb flows on $(Y,\xi)$ is formed the flows of the Reeb vector fields of contact forms on $(Y,\xi)$. Since in this article we will only consider co-oriented contact manifolds, we will for simplicity drop the term co-oriented. Thus, from now on, when we say contact manifold we actually mean co-oriented contact manifold.

The relation between contact topology and the topological entropy of Reeb flows was studied in \cite{Alves-Cylindrical,Alves-Anosov,A2,AlvesColinHonda2017,AlvesMeiwes2018,MacariniSchlenk2011}, where it is shown that there exist many examples of contact manifolds on which every Reeb flow has positive topological entropy. However, there are important examples of contact manifolds which admit Reeb flows with vanishing topological entropy. 
In dimension $3$ we have that $S^3$ endowed with its unique tight contact structure $\xi_0$ and $T^3$ and $\R P^3$ endowed with the contact structures associated to geodesic flows admit Reeb flows with vanishing topological entropy. Examples of such flows for these manifolds are, respectively, the periodic Reeb flow on $(S^3,\xi_0)$, the geodesic flow of the flat metric or of a metric of revolution on $ T^2$, and the geodesic flow of a metric of revolution on $S^2$.

More generally, $3$-dimensional pre-quantisation bundles admit periodic Reeb flows which have vanishing topological entropy.
In \cite[Section 7]{FHV} the authors apply Legendrian surgery to $3$-dimensional pre-quantization bundles and obtain Reeb flows with vanishing topological entropy on many contact $3$-manifolds.
A natural question which arises in such cases is whether there exist links formed by finitely many Reeb orbits that force topological entropy. In this article we show that such links do exist. 

To describe our results in more detail we introduce the following definition:
\begin{defn}
Let $(M,\xi)$ be a contact $3$-manifold that admits Reeb flows with vanishing topological entropy. A transverse link $L$ in $(M,\xi)$ is said to {\it force topological entropy} if every Reeb flow on $(M,\xi)$ which has $L$ as a set of Reeb orbits has positive topological entropy. 
\end{defn}

Our main structural results combine ideas introduced by the first author in \cite{Alves-Cylindrical,A2,thesis} with those of \cite{HMS,Momin} to give topological conditions on a transverse link that imply that it forces topological entropy. These conditions are formulated in terms of two topological invariants of a transverse link $L$:
\begin{itemize}
    \item[-] The cylindrical contact homology in the complement of the transverse link $L$, as defined by Momin in \cite{Momin}.
    \item[-] The strip Legendrian contact homology of a pair of  Legendrian knots $\Lambda$ and $\widehat{\Lambda}$ in the complement of the transverse link ${L}$, which is a relative version of Momin's theory and was introduced in \cite{thesis}.
\end{itemize}
\begin{rem}
A direct application of \cite[Theorem~2.6.12]{GeigesBook} shows that if a transverse link $L'$ belongs to the transverse isotopy class $[L]$ of $L$, then there exists a self contactomorphism of $(Y,\xi)$ which maps $L$ to $L'$. In particular, if $L$ forces topological entropy then any $L'\in [L]$ forces topological entropy as well.
\end{rem}

\begin{rem}
Given any transverse link $L$ on a contact $3$-manifold $(Y,\xi)$, there  always exist contact forms on $(Y,\xi)$ which have $L$ as a set of periodic orbits. This follows easily from the tubular neighbourhood theorem of Martinet for transverse knots in contact $3$-manifolds (see \cite[Theorem 2.5.15]{GeigesBook}), which says that any transverse knot in a contact $3$-manifold has a tubular neighbourhood where the contact structure is contactomorphic $(\D \times S^1,\ker(d\theta + xdy))$ for coordinates $\theta$ in $S^1$ and $(x,y)$ in $\D$. As a consequence, for any transverse link $L$ in a closed contact $3$-manifold admitting Reeb flows with vanishing $h_{top}$, it always makes sense to ask whether $L$ forces topological entropy.
\end{rem}

We can now state our main structural results. We start with
\begin{thm} \label{theorem:growthleg}
Let $(Y,\xi)$ be a closed contact $3$-manifold which admits Reeb flows with vanishing topological entropy. Let $L$ be a transverse link in $(Y,\xi)$ and $\lambda_0$ be a contact form on $(Y,\xi)$ adapted to $(Y\setminus L,\Lambda \to \widehat \Lambda)$ and such that  $LCH_L(\lambda_0,\Lambda \to \widehat \Lambda)$ has exponential homotopical growth. Then, the transverse link $L$ forces topological entropy in $(Y,\xi)$. 

Moreover, if $a>0$ denotes the exponential homotopical growth rate of $LCH_L(\lambda_0,\Lambda \to \widehat \Lambda)$, then for every contact form $\lambda$ on $(Y,\xi)$ which has $L$ as a set of Reeb orbits we have
\begin{equation}
    h_{top}(\phi_\lambda) \geq \frac{a}{\max f_\lambda},
\end{equation}
where $f_\lambda$ is the function such that $\lambda=f_\lambda \lambda_0$.
\end{thm}
The definition of $LCH_L$ and of the notion of a contact form adapted to $(Y\setminus L, \Lambda \to \widehat{\Lambda})$ are given in section \ref{section3}. The definition of the exponential growth rate of $LCH_L$ is presented in section \ref{section4}.

Our second structural result is
\begin{thm}\label{theorem:growth-entropy}
Let $(Y,\xi)$ be a closed contact $3$-manifold which admits Reeb flows with vanishing topological entropy. Let $L$ be a transverse link in $(Y,\xi)$ and $\lambda_0$ be a contact form on $(Y,\xi)$ such that 
\begin{itemize}
    \item $\lambda_0$ has $L$ as a set of periodic orbits and is hypertight on the complement of $L$,
    \item the cylindrical contact homology $CH_L(\lambda_0)$ has exponential homotopical growth.
\end{itemize}
Then, the transverse link $L$ forces topological entropy.

Moreover, if $a>0$ denotes the exponential homotopical growth rate of $CH_L(\lambda_0)$, then for every contact form $\lambda$ on $(Y,\xi)$ which has $L$ as a set of Reeb orbits we have
\begin{equation}
    h_{top}(\phi_\lambda) \geq \frac{a}{\max f_\lambda},
\end{equation}
where $f_\lambda$ is the function such that $\lambda=f_\lambda \lambda_0$.
\end{thm}
For a recollection of the definition of $CH_L$ and its properties, and for the definition of  its growth rate the reader should go to section \ref{section:growth-entropy}.

To justify the introduction of this technology we show how theorem \ref{theorem:growthleg} and theorem \ref{theorem:growth-entropy} can be used to prove the following result.
\begin{thm} \label{theorem:existence}
%Let $(M,\xi)$ be a closed $3$-dimensional contact manifold supported by an open book decomposition $(\Sigma,\psi,\Psi)$ with the properties that  $\partial \Sigma$ and the return map $\psi$ map is isotopic to a pseudo-Anosov map. Then there exists an open book decomposition $(\Sigma,\psi',\Psi')$ diffeomorphic to $(\Sigma,\psi,\Psi)$ that also supports $(M,\xi)$
%$\lambda_0$ adapted to the open book $(\Sigma,\psi,\Psi)$ such that $CH_{\mathcal B}(\lambda_0)$ has exponential homotopical growth. It follows that $\mathcal B$ forces topological entropy.

Let $(Y,\xi)$ be a closed $3$-dimensional contact manifold which admits Reeb flows with vanishing topological entropy. Then, there exist transverse knots in $(Y,\xi)$ that force topological entropy.

\end{thm}

%Our main results combine ideas of the first author \cite{A1,A2,A3} with results developed by Momin \cite{Momin} and  Hryniewicz, Momin and Salom\~ao \cite{HMS}. The results follow from the computation of a suitable Legendrian Contact Homology (LCH) on the complement of the transverse link. We count finite energy pseudo holomorphic strips whose boundary components lie on special Legendrian knots. These strips are asymptotic to Reeb chords connecting these Legendrian knots. One can filter the chain complex by the action of the asymptotic Reeb chords and the homotopy classes in the complement of the link. It turns out that the number of homotopy classes with non-vanishing LCH grows exponentially with the action. The positivity of topological entropy then follows from this exponential growth estimate and an application of Yomdin's theorem.
\subsection{Related developments}

\

\textbf{1.1.1 A Denvir-Mackay theorem for Reeb flows.} \\
In a joint work with Hryniewicz and Salom\~ao \cite{AHPS} we generalise to the category of Reeb flows a beautiful theorem of Denvir and Mackay which says that the geodesic flow of a Riemannian metric on $T^2$ with a contractible closed geodesic has positive topological entropy. To do this we use Theorem \ref{theorem:growth-entropy} to show that for any contractible flat knot $x$ in $T^2$, its transverse lift $L_x$ to the unit tangent bundle $(T_1T^2,\xi_{\rm geo})$ forces topological entropy. Here $\xi_{\rm geo}$ is the contact structure on $T_1T^2$ for which geodesic flows are Reeb flows, and $L_x$ is formed by the knots $(x^+,\frac{\dot{x}^+}{|\dot{x}^+|})$ and $(x^-,\frac{\dot{x}^-}{|\dot{x}^-|})$ where $x^+$ and $x^-$ are parametrizations of $x$ that induce opposite orientations. This gives an infinite family of explicit examples of distinct transverse links with two components in $(T_1T^2,\xi_{\rm geo})$ that force topological entropy.

\textbf{1.1.2 Forcing of topological entropy for positive contactomorphisms.} \\
It is natural to ask if one can generalise the results of this paper to the class of positive contactomorphisms. For this one would need to combine the ideas introduced here with the techniques developed by Dahinden, who used Rabinowitz-Floer homology to study the topological entropy of positive contactomorphisms in \cite{Dahinden2,Dahinden}.

\textbf{1.1.3 Recovering $h_{\rm top}$ with contact homologies.} \\
In this paper we focused on using the growth rate of contact homologies on the complement of transverse links as a tool to detect whether a transverse link forces topological entropy. However, we also plan to investigate how this tool can give information about the topological entropy of Reeb flows in a more general context. The following problem presents a direction of research which we plan to pursue: \\
{\it Given a Reeb flow $\phi_\lambda$ on a closed contact $3$-manifold $(M,\xi)$, does there exist a sequence $L_j$ of finite collections of Reeb orbits of $\phi_\lambda$ with the property that the exponential homotopical growth rates of $CH_{L_j}(\lambda)$ converge to $h_{\rm top}(\phi_\lambda)$ as $j \to +\infty$?} \\
In other words, how much of the topological entropy of Reeb flows can be recovered from the exponential growth rates of contact homologies?

\begin{rem}
The question of which contact $3$-manifolds admit Reeb flows with vanishing topological entropy is still wide open. The results in \cite{Alves-Cylindrical,Alves-Anosov,A2,AlvesColinHonda2017,MacariniSchlenk2011} exhibit large families of contact $3$-manifolds which do not admit Reeb flows with vanishing topological entropy, and it is natural to expect that most contact $3$-manifolds do not admit Reeb flows with vanishing $h_{\rm top}$. Other interesting results in this direction are obtained in \cite{CDR}, where the authors show that if $Y$ is a closed oriented hyperbolic $3$-manifold, then for any contact structure on $\xi$, every {\it non-degenerate} Reeb flow on $(Y,\xi)$ has positive $h_{\rm top}$. It is however still unclear if a closed hyperbolic contact $3$-manifold $(Y,\xi)$ can admit degenerate Reeb flows with vanishing $h_{\rm top}$, especially since most known examples of Reeb flows with vanishing $h_{\rm top}$ are indeed degenerate.
\end{rem}

\subsection{Organization of the paper}
In section \ref{section:background} we recall basic notions from contact geometry and dynamical systems.
In section \ref{section:blowup} we present a proof of a result of Bowen that says that blowing up a $3$-dimensional flow along periodic orbits does not change its topological entropy. This result is crucial for section \ref{section:growth-entropy} and is also interesting in its own right. In section \ref{section2} we recall the basic facts about pseudoholomorphic curves in symplectizations and symplectic cobordisms which we need to construct the Legendrian contact homology in the complement of a transverse link in section \ref{section3}. Section \ref{section4} contains the proof of theorem \ref{theorem:growthleg}. In section \ref{section:growth-entropy} we recall the basic facts about the cylindrical contact homology in the complement of a transverse link and prove theorem \ref{theorem:growth-entropy}. Finally, in section \ref{section:existence} we prove theorem \ref{theorem:existence}.

\

\textbf{Acknowledgements:} We thank Fr\'ed\'eric Bourgeois, Barney Bramham, Gerhard Knieper, Matthias Meiwes and Felix Schlenk for many helpful discussions and their interest in this work. Our special thanks to our collaborators Umberto Hryniewicz and Pedro A.S. Salom\~ao: our many discussions and our joint project \cite{AHPS} had a decisive influence in the development of this work, especially in sections \ref{section3} and \ref{section4}.

M.R.R. Alves was supported during the development of this project by the SFB/TRR 191 ``Symplectic Structures in Geometry, Algebra and Dynamics'' and by the ERC consolidator grant 646649  ``SymplecticEinstein''. A. Pirnapasov was supported during the development of this project by the SFB/TRR 191 ``Symplectic Structures in Geometry, Algebra and Dynamics''.

 %In section \ref{section:growth-entropy} we prove Theorem \ref{theorem:growth-entropy} which states that the exponential homotopical growth for the cylindrical contact homology on the complement of a transverse link, implies that the link forces topological entropy. \cite{Alves-Anosov,AlvesMeiwes2018,AlvesColinHonda2017}

\section{Recollections on contact geometry and dynamics} \label{section:background}

\subsection{Basics of contact geometry }
%We first recall some basic definitions from contact geometry. A 1-form $\lambda$  on a $3$-dimensional manifold $Y$ called a \textit{contact form} if $ \lambda \wedge d\lambda $ is a volume form on $Y$. The plane distribution $\xi= \ker \lambda$ is called the \textit{contact structure}. For us a \textit{contact $3$-manifold} will be a pair $(Y,\xi)$ such that $\xi$ is the kernel of some contact form $\lambda$ on $Y$ (these are usually called co-oriented contact manifolds in the literature). When $\lambda$ satisfies $\xi = \ker\lambda$, we will say that $\lambda$ is a contact form on $(Y,\xi)$. 

%The Reeb vector field associated to a contact form $\lambda$ is the unique vector field $R_\lambda$ on $Y$ satisfying $$\lambda(R_\lambda)\equiv 1 \qquad \mbox{ and } \qquad d\lambda (R_\lambda,\cdot) \equiv 0.$$ Its flow is called the Reeb flow of $\lambda$. The tangent bundle of $Y$  splits as $$TY = \ker d \lambda \oplus \ker \lambda= \R R_\lambda \oplus \xi.$$

Let $(Y,\xi)$ be a contact $3$-manifold and $\lambda$ a contact form on $(Y,\xi)$.
A periodic orbit $\gamma$ of the Reeb flow of $\lambda$ is called a Reeb orbit of $\lambda$. Its action $A(\gamma):=\int_{\gamma} \lambda$ coincides with its period $T$ since $\lambda(\dot \gamma)=\lambda(R_\lambda)=1.$

An embedded link in  $(Y,\xi)$ is called {\it Legendrian} if it is everywhere tangent to $\xi$. Given a contact form $\lambda$ and a pair of Legendrian knots $(\Lambda,\widehat{\Lambda})$, a \textit{Reeb chord} of $\lambda$ from $\Lambda$ to $\widehat{\Lambda}$ is a trajectory $\tau$ of the Reeb flow of $\lambda$ that starts in $\Lambda$ and ends in $\widehat{\Lambda}$. We define the action $A(\tau)$ of a Reeb chord $\tau$ as $A(\tau)=\int_\tau \lambda$. A Reeb chord $\tau$ is said to be \textit{transverse} if the intersection $\phi^{A(\tau)}_{X_\lambda}(\Lambda) \cap \widehat{\Lambda}$ is transverse at the endpoint of $\tau$.

A transverse link  $(Y,\xi)$ is an embedded link $L \hookrightarrow Y$ that is everywhere transverse to $\xi$.
%, and a transverse knot is a connected transverse link. 
The transverse isotopy class $[L]$ of a transverse link $L$ is the set of all transverse links in $(Y,\xi)$ which are isotopic to $L$ among transverse links.

\subsection{Topological entropy}

The topological entropy $h_{top}$ is a non-negative number that one associates to a dynamical system and which measures the complexity of the dynamics. Positivity of the topological entropy for a dynamical system implies some type of exponential instability.

We start with a definition of the topological entropy for flows on compact manifolds which is due to Bowen. Let $M$ be a closed manifold, $X$ be a $C^\infty$-smooth vector field on $M$ and $\phi$ be the flow of $x$. We consider an auxiliary Riemannian metric $g$ on $M$, and denote by $d_g$ the distance function associated to the metric $g$. Given positive numbers $T,\delta$ we say that a subset $S \subset M$  is  $T,\delta$-separated for $\phi$ if, for all points $p,q \in S$ with $p\neq q$, we have
$$\max_{t \in [0,T]}\{ d_g(\phi^t(p) ,\phi^t(q)) \} > \delta.  $$ 
We let $n^{T,\delta}_X$ be the maximal cardinality of a $T,\delta$-separated set for the flow $\phi$ of $X$. The $\delta$-entropy $h_\delta$ is then defined by 
$$h_\delta(\phi):= \limsup_{T\to +\infty} \frac{\log(n^{T,\delta}_X)}{T},$$
and the topological entropy $h_{top}$ is defined by
\begin{equation}
    h_{top}(\phi):= \lim_{\delta \to 0} h_{\delta}(\phi).
\end{equation}
Geometrically, we see that $h_\delta$ measures the exponential growth rate of the number of orbits which are distinguishable with precision $\delta$ as time advances: $h_{top}$ is then the limit of these growth rates as the $\delta$ goes to $0$.
We refer the reader to \cite{Hasselblatt-Katok} for the basic properties of $h_{top}$.

By deep results of Yomdin and Newhouse the topological entropy of a $C^\infty$-smooth flow $\phi$ coincides with the volume growth ${\rm vol}(\phi)$ of submanifolds by $\phi$, which we now define. For the auxiliary Riemannian metric $g$, we let ${\rm Vol}_g^k$ denote the $k$-dimensional volume with respect to $g$.
Then for any closed $k$-dimensional $C^\infty$-submanifold $V$ of $M$, we define 
$${\rm vol}(V,\phi):= \limsup_{t \to +\infty} \frac{\log ({\rm Vol}_g^k(\phi^t(V)) )}{t}.$$
We see that ${\rm vol}(V,\phi)$ measures the exponential growth as $t \to +\infty$ of the volume ${\rm Vol}_g^k(\phi^t(V))$ of the image of $V$ by $\phi^t$. The volume growth ${\rm vol}(\phi)$ is then defined by
\begin{equation}
{\rm vol}(\phi) := \sup_{V \in {\rm Sub}^{\infty}(M)} \{{\rm vol}(V,\phi))  \},
\end{equation}
where ${\rm Sub}^{\infty}(M)$ denotes the set of $C^\infty$-smooth closed submanifolds of $M$ of all possible co-dimensions. 

In this article we use two different techiniques to establish positivity of topological entropy for Reeb flows. The technique used in section \ref{section4} consists in proving that the volume growth of certain Legendrian knots by the Reeb flows is exponential. It then follows from the results of Yomdin that the topological entropy of the flow is positive. The technique used in section \ref{section:growth-entropy} consists in using the exponential growth of the number of periodic orbits in different homotopy classes to establish the exponential growth of $n^{T,\delta}_X$, for sufficiently small $\delta$.

One motivation for studying the topological entropy of $3$-dimensional Reeb flows is that positivity of $h_{top}$ for such a flow implies that it has a rich orbit structure. This follows from the following fundamental result of Katok which is found in supplement S.5 of \cite{Hasselblatt-Katok}.

\begin{thm*} [Katok]
If $\phi$ is a smooth flow on a closed oriented 3-manifold generated by a non-vanishing vector field, then $\phi$ has positive topological entropy if, and only if, there exists a “horseshoe” as a subsystem of the flow. As a consequence, the number of hyperbolic periodic orbits of $\phi$ grows exponentially with respect to the period.
\end{thm*}
Here a “horseshoe” denotes a compact invariant set where the dynamics is semi-conjugate to that of the suspension of a subshift of finite type by a finite-to-one map.

\section{Topological entropy and blow up along periodic orbits} \label{section:blowup}

Let $Y$ be a closed $3$-manifold, $X$ a smooth non-vanishing vector field on $Y$ and $\phi^t$ its flow. Let $\mathcal{P}=\{P_1,...,P_n\}$ be a finite collection of distinct simple periodic orbits of $\phi^t$.
We denote by $T_j$ the primitive period of the orbit $P_j$.

We consider the blow up of the manifold $Y$ along the elements of $\mathcal{P}$; our description follows that of \cite{Umberto1}. For this we first choose for each $1\leq j \leq n$ a tubular neighbourhood $V_j$ of $P_j$ and an orientation preserving diffeomorphism
\begin{equation}
\Psi_j: V_j \to \mathbb{R}/T_j\mathbb{Z} \times \mathbb{D},
\end{equation}
such that $\Psi_j(\phi^t(p))=(t,0)$ for all $p\in P_j$. On $V_j \setminus P_j$ we have tubular polar coordinates $(t,r,\theta) \in \mathbb{R}/T_j\mathbb{Z} \times (0,1] \times \mathbb{R}/2\pi\mathbb{Z}$ coming from the identification $\Psi^{-1}(t,re^{i\theta})\simeq (t,r,\theta)$.

We then define
\begin{equation}
Y_{\mathcal{P}}:= \bigg\{ (Y \setminus \mathcal{P}) \sqcup  \big(\bigsqcup_{j=1}^n \mathbb{R}/T_j\mathbb{Z} \times [0,1] \times \mathbb{R}/2\pi\mathbb{Z}\big)   \bigg\},
\end{equation}
where, for each $1\leq j\leq n$, we identify the points $(t,r,\theta)\in \mathbb{R}/T_j\mathbb{Z} \times [0,1] \times \mathbb{R}/2\pi\mathbb{Z}$ with $\Psi_j^{-1}(t,re^{i\theta})\in V_j \setminus P_j$. We remark that $ Y_{\mathcal{P}}$ is a compact $3$-manifold with boundary, and that $\partial  Y_{\mathcal{P}}$ is a collection of $2$-dimensional tori $T_j$, each $T_j$ being the blow up of the orbit $P_j$.

Let $\Pi: Y_{\mathcal{P}} \to Y$ be the smooth projection which is the extension to all of $Y_{\mathcal{P}}$ of the identity map $i: Y \setminus \mathcal{P} \to Y \setminus \mathcal{P} $ where the domain is seen as a subset of $ Y_{\mathcal{P}}$ and the target as a subset of $ Y$. It is clear that $\Pi$ is surjective.

As shown in \cite{Umberto1}, there is a unique well-defined smooth non-vanishing vector field $\widehat{X}$ on $ Y_{\mathcal{P}}$ which projects via $\Pi$ to the vector-field $X$ on $Y$. Since the $P_j$ are periodic orbits of $X$, it follows that $\widehat{X}$ is tangent to the tori $T_j$. We denote by $\widehat{\phi}^t$ the flow of $\widehat{X}$.

The following result is due to Bowen \cite{Bowen}. 
\begin{thm} \label{theorem-Bowen}
For the flows $\phi^t$ and $\widehat{\phi}^t$ we have $$h_{\rm top}(\phi^t)= h_{\rm top}(\widehat{\phi}^t).$$
\end{thm}
\proof: The flows  $\phi^t$ and $\widehat{\phi}^t$ satisfy $\Pi\circ \widehat{\phi}^t =\phi^t \circ \Pi $. Because $\Pi$ is surjective we conclude that $\phi^t$ is semi-conjugate to $\widehat{\phi}^t$ via $\Pi$, and it is well-know that this implies  $h_{\rm top}(\phi^t) \leq h_{\rm top}(\widehat{\phi}^t).$

To obtain the reverse inequality one proceeds as follows. By the variational principle for topological entropy we have $$h_{\rm top}(\widehat{\phi}^t) = \sup_{\mu \in \mathfrak{M}_{\rm erg}(\widehat{\phi}^t)} H(\widehat{\phi}^t, Y_{\mathcal{P}},\mu),$$ where of $H(\widehat{\phi}^t, Y_{\mathcal{P}},\mu)$ denotes the measure theoretical entropy of $\widehat{\phi}^t$ with respect to a $\widehat{\phi}^t$-invariant probability measure $\mu$, and $\mathfrak{M}_{\rm erg}(\widehat{\phi}^t)$ denotes the set of $\widehat{\phi}^t$-ergodic measures. 

If $h_{\rm top}(\widehat{\phi}^t)=0$ then there's nothing to be proved. So we assume that this is not the case and, applying the variational principle for $h_{\rm top}$, let $\mu \in \mathfrak{M}_{\rm erg}(\widehat{\phi}^t)$ be such that $H(\widehat{\phi}^t, Y_{\mathcal{P}},\mu)>0$. 

The tori $T_j$ for $j\in \{1,...,n\}$ are  invariant by $\widehat{\phi}^t$. We claim that $\mu(T_j) = 0$ for every $j\in \{1,...,n\}$. Fix $j\in \{1,...,n\}$.
Since $\mu$ is $\widehat{\phi}^t$-ergodic and $T_j$ is an invariant set of $\widehat{\phi}^t$ we know that $\mu(T_j)$ is either $0$ or $1$. If we had $\mu(T_j)=1$ we would conclude that the support of $\mu$ is $T_j$. This implies that $H(\widehat{\phi}^t, Y_{\mathcal{P}},\mu)$ equals the measure theoretical entropy $H(\widehat{\phi}^t, T_j,\mu)$ of the restriction of $\widehat{\phi}^t$ to $T_j$ with respect to the probability measure $\mu$ restricted to $T_j$.
Since $\widehat{X}$ restricted to $T_j$ is a smooth non-vanishing vector-field on $T_j$ we know from \cite{Young} that $h_{\rm top}(\widehat{\phi}_{T_j}^t)=0$, which implies that $H(\widehat{\phi}^t, T_j,\mu)=0$. But this contradicts our assumption that   $H(\widehat{\phi}^t, Y_{\mathcal{P}},\mu)=H(\widehat{\phi}^t, T_j,\mu) $ is positive.
It follows that $\mu(T_j) = 0$ for every $j\in \{1,...,n\}$.

Since $\mu(\partial Y_{\mathcal{P}})=0 $ we conclude that $H(\widehat{\phi}^t, Y_{\mathcal{P}},\mu)= H(\widehat{\phi}^t, Y_{\mathcal{P}}\setminus \partial Y_{\mathcal{P}},\mu)$. Let $\nu$ be the probability measure on $Y \setminus \mathcal{P}$ obtained by pulling back $\mu$ by the map $\Pi^{-1}: Y \setminus \mathcal{P}\to Y_{\mathcal{P}}\setminus \partial Y_{\mathcal{P}} $. Clearly $\nu$ extends to a Borel probability measure on $Y$ such that $\nu(\mathcal{P})=0$.

Since $\Pi\circ \widehat{\phi}^t =\phi^t \circ \Pi $ and the measure $\nu$ is $\phi^t$-invariant, the measurable dynamical systems $(\widehat{\phi}^t, Y_{\mathcal{P}}\setminus \partial Y_{\mathcal{P}},\mu)$ and $(\phi^t, Y \setminus \mathcal{P}, \nu)$ are isomorphic, which implies that $H(\widehat{\phi}^t, Y_{\mathcal{P}}\setminus \partial Y_{\mathcal{P}},\mu) =(\phi^t, Y \setminus \mathcal{P}, \nu) $. Since $\nu(\mathcal{P})=0$ it follows that $H(\phi^t, Y \setminus \mathcal{P}, \nu) = H(\phi^t, Y, \nu)$. We then finally conclude that $ H(\phi^t, Y, \nu) = H(\widehat{\phi}^t, Y_{\mathcal{P}},\mu)$. 

Summing up we have, assuming that $h_{\rm top}(\widehat{\phi}^t)>0$, produced for each $\mu \in \mathfrak{M}_{\rm erg}(\widehat{\phi}^t)$ such that $H(\widehat{\phi}^t, Y_{\mathcal{P}},\mu)>0$ a $\phi^t$ invariant measure $\nu$ such that $ H(\phi^t, Y, \nu) = H(\widehat{\phi}^t, Y_{\mathcal{P}},\mu)$. Applying the variational principle for $h_{\rm top}(\widehat{\phi}^t)$ and $h_{\rm top}({\phi}^t)$ we conclude that $h_{\rm top}(\widehat{\phi}^t)\leq h_{\rm top}({\phi}^t)$ which completes the proof of the theorem. \qed

\section{Pseudoholomorphic curves}\label{section2}

Contact homology on the complement of Reeb orbits is defined in the spirit of the SFT-invariants introduced in~\cite{SFT}. To define it we use pseudoholomorphic curves in symplectizations and symplectic cobordisms. Pseudoholomorphic curves were introduced in symplectic manifolds by Gromov~\cite{Gr}, and in symplectizations by Hofer~\cite{H}; see also~\cite{CPT}.
%as a general reference for pseudo holomorphic curves in symplectic cobordisms.

\subsection{Curves in symplectizations and symplectic cobordisms}

\subsubsection{Cylindrical almost complex structures} Let $(Y,\xi)$ be a smooth contact $3$-manifold and $\lambda$ a contact form on $(Y,\xi)$. The symplectization of $(Y,\xi=\ker \lambda)$ is the product $\mathbb{R} \times Y$ equipped with the symplectic form $d(e^s \lambda)$, where $s$ denotes the $\R$-coordinate on $\mathbb{R} \times Y$.

The 2-form $d\lambda$ restricts to a symplectic form on the vector bundle $\xi\to Y$ and it is well known that the set  $\mathfrak{j}(\lambda)$ of $d\lambda$-compatible complex structures on $\xi$ is non-empty and contractible in the $C^\infty$-topology. The $d\lambda$-compatibility of $j\in\mathfrak{j}(\lambda)$ means that $d\lambda (\cdot, j \cdot)$ is a positive-definite inner product on $\xi$ and hence 
\begin{equation}\label{metric} 
\left<u,v\right>_j:=\lambda(u)  \lambda(v) + d\lambda(\pi u, j \pi v) \qquad u,v \in TY
\end{equation} 
is a Riemannian metric on $Y$. Here $\pi:TY \to \xi$ stands for the projection  along the Reeb vector field. Note that $\mathfrak{j}(\lambda)$ depends, in fact, only on the co-oriented contact structure $\xi$, whose orientation is induced by $d\lambda$.

For $j \in \mathfrak{j}(\lambda)$ one defines an $\mathbb{R}$-invariant almost complex structure $J$ on $\mathbb{R} \times Y$ by demanding that
\begin{equation} \label{eq16}
 J \cdot \partial_s = R_\lambda \mbox{ and }  J|_\xi = j.
\end{equation} One checks that $J$ is $d(e^s \lambda)$-compatible. We denote by $\mathcal{J}(\lambda)$ the space of almost complex structures $J$ on $\mathbb{R} \times Y$ which satisfy \eqref{eq16} for some $j \in \mathfrak{j}(\lambda)$.

\subsubsection{Exact symplectic cobordisms}\label{sec_exact}

Let $\lambda^+$ and $\lambda^-$ be contact forms on the contact manifold $(Y,\xi)$. There exists a smooth function $f:Y \to (0,+\infty)$ satisfying $\lambda^+ = f \lambda^-$. Assume that $f>1$ pointwise. 

Choose $\chi:\R \times Y \to \R$ smooth satisfying
\begin{equation}\label{chi_symp_form}
\begin{aligned}
& \chi  = e^sf \mbox{ on } [1,+\infty)\times Y,\\
& \chi  = e^s  \mbox{ on } (-\infty,0]\times Y,\\
& \partial_s\chi  > 0 \mbox{ on } \R \times Y,
\end{aligned}
\end{equation}
where $s$ is the $\R$-coordinate.  It then follows that 
\begin{equation}\label{exact_symp_form_on_cob}
\varpi:= d(\chi \lambda_-)
\end{equation}
is a symplectic form on $\R \times Y$.

%\begin{lem}
%Let $\chi_0,\chi_1$ be smooth functions satisfying~\eqref{chi_symp_form}. There is a smooth function $\varphi:\R\times Y\to \R$ such that $\chi_1(\varphi(s,p),p)=\chi_0(s,p)$ for all $(s,p)\in\R\times Y$. In particular,~$\varphi$ satisfies $\varphi(s,p)=s$ on $(\R\setminus[0,1])\times Y$ and $\partial_s\varphi>0$ on $\R\times Y$.
%\end{lem}

%\begin{proof}
%For all $p\in Y$ and $j\in\{0,1\}$ the map $s\mapsto\chi_j(s,p)$ defines an increasing diffeomorphism $\R\to(0,+\infty)$. Hence we can define $\varphi(s,p)$ by requiring that the identity $\chi_1(\varphi(s,p),p)=\chi_0(s,p)$ holds for all $(s,p)$. Smoothness of $\varphi$ follows from the fact that each $\chi_j(\cdot,p)$ is a smooth family of diffeomorphisms $\R \to (0,+\infty)$ parametrized by $p\in Y$.
%\end{proof}

%\begin{cor}\label{cor_exact_symp_forms}
%Let $\chi_0,\chi_1$ satisfy~\eqref{chi_symp_form} and let $\varphi$ be given by applying the previous lemma to $\chi_0,\chi_1$. Then the diffeomorphism $\Psi:\R\times Y \to \R\times Y$, $\Psi(s,p)=(\varphi(s,p),p)$ satisfies $\Psi^*d(\chi_1\lambda_-)=d(\chi_0\lambda_-)$.
%\end{cor}

%\begin{proof}
%Direct consequence of the easily verifiable identity $\Psi^*(\chi_1\lambda_-)=\chi_0\lambda_-$.
%\end{proof}

We call $(\R \times Y, \varpi)$ an exact symplectic cobordism from $\lambda_+$ to $\lambda_-$. On such a cobordism $(\R \times Y, \varpi)$ we consider almost complex structures $\bar J$ defined as follows: fix any two choices $J^+ \in \mathcal{J}(\lambda^+)$, $J^- \in \mathcal{J}(\lambda^-)$ and choose $\bar J$ satisfying
$$
\begin{aligned}
\bar J & =  J^+   \mbox{ on }  [1,+\infty)\times Y,  \\
\bar J & =  J^-   \mbox{ on }  (-\infty,0]\times Y,  \\
\bar J & \mbox{ is compatible with } \varpi \mbox{ on } [0,1]\times Y.
\end{aligned}
$$
The space of such almost complex structures $\bar J$ is denoted by $\mathcal{J}_\varpi(J^+,J^-)$. Finally we note that standard arguments will show that $\mathcal{J}_\varpi(J^+,J^-)$ is non-empty and contractible in the $C^\infty$-topology.
%It follows from Corollary~\ref{cor_exact_symp_forms} that if $\varpi_0,\varpi_1$ are two symplectic forms as in~\eqref{exact_symp_form_on_cob} then there is a diffeomorphism $\Psi$ of $\R\times Y$ supported in $[0,1]\times Y$ such that the map $\bar J \mapsto \Psi^*{\bar J}$ establishes a bijection between $\mathcal{J}_{\varpi_0}(J^+,J^-)$ and $\mathcal{J}_{\varpi_1}(J^+,J^-)$. Later it will be important to notice that $\Psi$ can be chosen to preserve the $Y$-component. 

\subsubsection{Splitting families} \label{splitting}

Let $\lambda^+$, $\lambda$ and $\lambda^-$ be contact forms on a closed contact $3$-manifold $(Y,\xi)$. Writing $\lambda_+=f_+ \lambda$ and $\lambda_-=f_- \lambda$ for smooth functions $f_+,f_-:Y \to (0,+\infty)$ we assume that $f_+>1 > f_-$ pointwise.

Fix $J_+ \in \mathcal{J}(\lambda^+), J \in \mathcal{J}(\lambda)$ and $J_- \in \mathcal{J}(\lambda^-)$. Let $(\R\times Y,\varpi_+)$ be an exact symplectic cobordism from $\lambda_+$ to $\lambda$ and let $(\R\times Y,\varpi_-)$ be an exact symplectic cobordism from $\lambda$ to $\lambda_-$, as defined in the previous paragraph. Next we choose $\bar J_+ \in \mathcal{J}_{\varpi_+}(J^+,J)$ and $\bar J_- \in \mathcal{J}_{\varpi_-}(J,J^-)$ and then define $\widehat J_R$ by

\begin{equation}\label{split}
\begin{aligned}
\widehat J_R &= (T_{-R})^*\bar J_+\mbox{ on } [0,+\infty)\times Y,\\
\widehat J_R &= (T_{R+1})^*\bar J_-\mbox{ on }  (-\infty,0]\times Y,\\
\end{aligned}
\end{equation} where $T_R(s,p):=(s+R,p), \ \forall (s,p)\in \R \times Y$. By definition, $\widehat J_R$ is smooth and agrees with $J$ on the neck region $[-R,R]\times Y$.

\subsubsection{Pseudoholomorphic curves} \label{pseudohol}

In this section $(Y,\xi)$ denotes a closed contact $3$-manifold and $\Lambda \subset Y$ a Legendrian link. If $\lambda$ is a contact form on $(Y,\xi)$ then $\R\times\Lambda$ is an exact Lagrangian in $\R\times Y$ with respect to any symplectic form $d(h\lambda)$ where $h:\R\times Y\to \R$ satisfies $h>0$, $\partial_sh>0$.

Let $(S,j)$ be a compact Riemann surface (possibly with boundary) and let $J \in \mathcal{J}(\lambda)$. A finite-energy pseudoholomorphic curve in $(\R \times Y,J)$ with boundary in $\R\times\Lambda$ is a smooth map $$ \wtil=(a,w): S \setminus \Gamma \to \mathbb{R} \times Y $$ where $\Gamma\subset S$ is a finite set, that satisfies
$$\begin{aligned}
& \bar \partial_J \wtil := \frac{1}{2}\left(d\wtil + J(\wtil) \circ d\wtil \circ j\right)=0, \\
& \wtil(\partial S \setminus \Gamma) \subset \R\times\Lambda,
\end{aligned}$$
and has finite Hofer energy
\begin{equation}
0 < E(\widetilde{w}):= \sup_{q \in \mathcal{E}} \int_{S \setminus \Gamma} \widetilde{w}^*d(q\lambda)<+ \infty.
\end{equation}
Here  
\begin{equation}\label{energia} 
\mathcal{E}= \{ q: \mathbb{R} \to [0,1] \mbox{ smooth }; q' \geq 0\}.
\end{equation}

Let $\lambda_+,\lambda_-$ be contact forms on $(Y,\xi)$ and consider an exact symplectic cobordism $(W=\mathbb{R} \times Y, \varpi)$ from $\lambda^+$ to $\lambda^-$ of the kind defined in~\ref{sec_exact}. Consider $\bar J \in \mathcal{J}_\varpi(J^+,J^-)$ where $J^+ \in \mathcal{J}(\lambda^+)$ and $J^- \in \mathcal{J}(\lambda^-)$. A finite-energy pseudoholomorphic curve with boundary in $\R\times\Lambda$ is a smooth map $$ \wtil=(a,w): S \setminus \Gamma \to \R \times Y $$ where $\Gamma \subset S$ is a finite set, that satisfies
$$
\begin{aligned}
& \bar \partial_{\bar J} \wtil := \frac{1}{2}\left(d\wtil + \bar J(\wtil) \circ d\wtil \circ j\right)=0, \\
& \wtil(\partial S \setminus \Gamma) \subset \R\times\Lambda,
\end{aligned}
$$
and has finite energy
\begin{equation}\label{encob}
0<E_{\lambda^+}(\wtil) + E_c(\wtil) + E_{\lambda^-} (\wtil) < +\infty
\end{equation}
where
$$
\begin{aligned}
&E_{\lambda^+}(\wtil) = \sup_{q \in \mathcal{E}} \int_{\wtil^{-1}([1,+\infty)\times Y)} \wtil^*d(q\lambda^+) \\
&E_c(\wtil) = \int_{\wtil^{-1}([0,1]\times Y)} \wtil^*\varpi \\
&E_{\lambda^-}(\wtil) = \sup_{q \in \mathcal{E}} \int_{\wtil^{-1}((-\infty,0]\times Y)} \wtil^*d(q\lambda^-)
\end{aligned}
$$

Let $\widehat J_R$ be a splitting almost complex structure as in Section~\ref{splitting}. These are defined by~\eqref{split}, where $J_\pm\in\J(\lambda_\pm)$, $J\in\J(\lambda)$, $\bar J_+ \in \mathcal{J}_{\varpi_+}(J^+,J)$ and $\bar J_- \in \mathcal{J}_{\varpi_-}(J,J^-)$. Finite-energy $\widehat J_R$-curves $\wtil$ are defined analogously but we need a modified version of the finite-energy condition:
\begin{equation}
0<E_{\lambda^+} (\wtil)+ E_{\lambda,\lambda^+}(\wtil) + E_{\lambda}(\wtil)+ E_{\lambda^-,\lambda}(\wtil)+ E_{\lambda^-}(\wtil)  < +\infty
\end{equation}
where
$$
\begin{aligned}
&E_{\lambda^+}(\wtil) = \sup_{q \in \mathcal{E}} \int_{\wtil^{-1}([R+1,+\infty)\times Y)} \wtil^*d(q\lambda^+) \\
&E_{\lambda,\lambda^+}(\wtil) =\int_{\wtil^{-1}([R,R+1]\times Y)} \wtil^*(T_{-R})^*\varpi_+ \\
&E_{\lambda}(\wtil) = \sup_{q \in \mathcal{E}} \int_{\wtil^{-1}([-R,R]\times Y)} \wtil^*d(q\lambda) \\
&E_{\lambda^-,\lambda}(\wtil) = \int_{\wtil^{-1}([-R-1,-R]\times Y)} \wtil^*(T_{R+1})^*\varpi_- \\
&E_{\lambda^-}(\wtil) = \sup_{q \in \mathcal{E}} \int_{\wtil^{-1}((-\infty,-R-1]\times Y)} \wtil^*d(q\lambda^-) \\
\end{aligned}
$$ 
See~\eqref{energia} and Section~\ref{splitting} for definitions.

The elements of $\Gamma \subset S$ are called punctures of $\wtil$. We call elements of $\Gamma_\partial := \Gamma \cap \partial S$ boundary punctures, and elements of $\Gamma \setminus \Gamma_\partial$ interior punctures. Assume that all punctures are non-removable. We only describe the behaviour of finite-energy pseudoholomorphic curves near non-removable punctures on exact symplectic cobordisms, since the behaviour in the cases of symplectizations and splitting symplectic cobordisms is the same.

According to \cite{Ab,H,HWZ} the punctures are classified in four different types. Before presenting this classification we introduce some notation:
\begin{itemize}
\item If $z\in \Gamma \setminus \Gamma_\partial$ then a neighbourhood of $z$ in $(S,j)$ is bi-holomorphic to the unit disk $\mathbb{D}\subset\C$ in such a way that $z\equiv 0$. Then $\mathbb{D}\setminus \{0\}$ is bi-holomorphic to $[0,+\infty) \times \R / \Z$ via the map $s+it \mapsto e^{-2\pi(s+it)}$. We always write $\wtil=\wtil(s,t)$ near $z\in \Gamma \setminus \Gamma_\partial$ using these exponential holomorphic coordinates $(s,t)$ in $[0,+\infty)\times \R/ \Z$.
\item If $z\in \Gamma_\partial$ then a neighbourhood of $z$ in $(S,j)$ is bi-holomorphic to the half-disk $\D^+ = \{|z|\leq 1,\Im(z) \geq 0\}$ with $z \equiv 0$ and $\D^+ \setminus \{0\}$ is bi-holomorphic to $([0,+\infty) \times [0,1],i_0)$ via the map $s+it \mapsto -e^{-\pi(s+it)}$. We always write $\wtil=\wtil(s,t)$ near $z\in \Gamma_\partial$ using these exponential holomorphic coordinates $(s,t)$ in $[0,+\infty)\times [0,1]$.
\end{itemize}
The non-removable punctures of a finite energy pseudoholomorphic curve $\wtil=(a,w)$ are classified as follows:
\begin{itemize}
\item $z \in \Gamma \setminus \Gamma_\partial$ is a positive interior puncture, i.e., $a(z')\to + \infty$ as $z' \to z$ and given a sequence $s_n \to +\infty$ there exists a subsequence, still denoted $s_n$, and a periodic Reeb orbit $\gamma^+$ of $\lambda^+$ with period $T^+$ such that $w(s_n,\cdot)$ converges in $C^\infty$ to $\gamma^+(T^+\cdot)$ as $n\to +\infty$.
\item $z \in \Gamma \setminus \Gamma_\partial$ is a negative interior puncture, i.e., $a(z')\to - \infty$ as $z' \to z$ and given a sequence $s_n \to +\infty$ there exists a subsequence, still denoted $s_n$, and a periodic Reeb orbit $\gamma^-$ of $\lambda^-$ with period $T^-$ such that $w(s_n,\cdot)$ converges in $C^\infty$ to $\gamma^-(-T^-\cdot)$ as $n\to +\infty$.
\item $z \in \Gamma_\partial$ is a positive boundary puncture, i.e.,  $a(z')\to + \infty$ as $z' \to z$ and given a sequence $s_n \to + \infty$ there exists a subsequence, still denoted $s_n$, and a Reeb chord $\tau^+$ of  $\lambda^+$ from $\Lambda$ to itself with action $l^+$ such that $w(s_n,\cdot)$ converges in $C^\infty$ to $\tau^+(l^+ \cdot)$ as $n\to +\infty$.
\item $z \in \Gamma_\partial$ is a negative boundary puncture, i.e.,  $a(z')\to - \infty$ as $z' \to z$ and given a sequence $s_n \to + \infty$ there exists a subsequence, still denoted $s_n$, and a Reeb chord $\tau^-$ of  $\lambda^-$ from $\Lambda$ to itself with action $l^-$ such that $w(s_n,\cdot)$ converges in $C^\infty$ to $\tau^-(-l^-\cdot)$ as $n\to +\infty$.
\end{itemize}

Denote by $\T(\lambda,\Lambda)$ the set of Reeb chords of $\lambda$ from $\Lambda$ to itself. Recall that a chord $\tau \in \T(\lambda,\Lambda)$ with action $\ell$ is called transverse if the linearized flow in time $\ell$ maps $T_{\tau(0)}\Lambda$ to a complement of $T_{\tau(\ell)}\Lambda$ in $\xi$. If the Legendrian link $\Lambda = \Lambda_0 \cup \Lambda_1$ is written as a union of (not necessarily disjoint) Legendrian links $\Lambda_0,\Lambda_1$ then we denote by $\T(\lambda, \Lambda_0 \to \Lambda_1) \subset \T(\lambda,\Lambda)$ the set of Reeb chords from $\Lambda_0$ to $\Lambda_1$.

Denote by $\P(\lambda)$ the set of periodic Reeb orbits, where geometrically identical periodic orbits with distinct periods are distinguished. We may think of elements of $\P(\lambda)$ as equivalence classes of pairs $(\gamma,T)$ where $\gamma:\R\to Y$ is a $T$-periodic trajectory of $R_\lambda$, and two pairs $(\gamma_0,T_0)$, $(\gamma_1,T_1)$ are equivalent if, and only if, $\gamma_0(\R)=\gamma_1(\R)$ and $T_0=T_1$. A periodic orbit $\gamma = (\gamma,T) \in \P(\lambda)$ is called non-degenerate if the transverse Floquet multipliers of $\gamma$ in period $T$ are not equal to $1$.

\begin{defn}
For a boundary (interior) puncture $z$, if there is a sequence $s_n$ such that  $w(s_n,\cdot)$ converges to a given Reeb chord $\tau \in \T(\lambda,\Lambda)$ (or to a given periodic Reeb orbit~$\gamma \in \P(\lambda)$), we will say that $\tau$ (or $\gamma$) is an asymptotic limit of $\wtil$ at $z$. If all asymptotic limits at $z$ coincide with a certain chord $\tau$ (or with a certain periodic orbit $\gamma$) then we say that $\wtil$ is asymptotic to $\tau$ (or $\gamma$).
\end{defn}

\begin{rem}
If a non-degenerate periodic Reeb orbit $\gamma$ or a transverse Reeb chord $\tau$ is an asymptotic limit of $\wtil$ at a puncture $z$ then $\wtil$ is asymptotic at $z$ to $\gamma$ or $\tau$. This follows from results of~\cite{Ab,HWZ} where a much more detailed description of the convergence is given.
\end{rem}

For curves with boundary punctures we are particularly interested in the case where $(S \setminus \Gamma,j)$ is biholomorphic to $\mathbb{R} \times [0,1]$ with its standard complex structure. We call $\wtil$ a pseudoholomorphic strip. In this case the domain is equivalent to a disk with two punctures on the boundary.

Let $J\in \mathcal{J}(\lambda)$. Suppose that
%the a finite-energy strip $\wtil$ satisfies 
%\begin{equation}\label{boundary_condition_strips}
%\wtil(\R \times \{0\}) \subset \mathbb{R}\times \Lambda_0, \qquad\qquad \wtil(\R \times \{1\}) \subset \R \times \Lambda_1
%\end{equation}
%where 
$\Lambda = \Lambda_0 \cup \Lambda_1$ is written as a union of two disjoint Legendrian links $\Lambda_0,\Lambda_1$. For fixed Reeb chords $\tau_+,\tau_- \in \T(\lambda, \Lambda_0 \to \Lambda_1)$ we denote by $\mathcal{M}(J, \tau_+,\tau_-)$ the moduli space whose elements are equivalence classes of finite-energy pseudoholomorphic strips $\wtil: \R\times[0,1] \to (\mathbb{R}\times Y,J)$ such that
\begin{itemize}
\item $\wtil$ is positively asymptotic to $\tau_+$ at $+\infty \times [0,1]$;
\item $\wtil$ is negatively asymptotic to $\tau_-$ at $-\infty \times[0,1]$;
\item $\wtil$ satisfies
\begin{equation}\label{boundary_condition_strips}
\wtil(\R \times \{0\}) \subset \mathbb{R}\times \Lambda_0, \qquad\qquad \wtil(\R \times \{1\}) \subset \R \times \Lambda_1.
\end{equation}
\end{itemize}
Two such pseudoholomorphic strips $\wtil(s,t)$ and $\wtil'(s,t)$ represent the same element in $\mathcal{M}(J, \tau_+,\tau_-)$ if and only if $\wtil'(s,t) = \wtil(s+s_0,t)$ for some $s_0\in \R$.

Analogously, for fixed periodic orbits $\gamma_+,\gamma_-\in\P(\lambda)$ we denote by $\M(J,\gamma_+,\gamma_-)$ the moduli space whose elements are equivalence classes of finite-energy pseudoholomorphic cylinders $\wtil: \R\times\R/\Z \to (\mathbb{R}\times Y,J)$ such that
\begin{itemize}
\item $\wtil$ is positively asymptotic to $\gamma_+$ at $+\infty \times \R/\Z$;
\item $\wtil$ is negatively asymptotic to $\gamma_-$ at $-\infty \times \R/\Z$.
\end{itemize}
Here, and also in what follows, $\R\times\R/\Z$ is endowed with its standard complex structure. Two such pseudoholomorphic cylinders $\wtil(s,t)$ and $\wtil'(s,t)$ represent the same element in $\mathcal{M}(J,\gamma_+,\gamma_-)$ if and only if $\wtil'(s,t) = \wtil(s+s_0,t+t_0)$ for some $(s_0,t_0)\in \R\times\R/\Z$.

The moduli spaces $\M(J,\tau_+,\tau_-)$, $\M(J,\gamma_+,\gamma_-)$ admit an $\R$-action given by translations in the first component of $\R\times Y$. This is so because both $J$ both and $\R\times \Lambda$ are $\R$-invariant.

\begin{lem}\label{lem_moduli_action}
The following holds:
\begin{itemize}
\item If $C \in \M(J,\gamma_+,\gamma_-)$ is fixed by translation by a non-zero number then $\gamma_+=\gamma_-$ and $C$ is a cylinder over a periodic orbit.
\item If $C \in \M(J,\tau_+,\tau_-)$ is fixed by translation by a non-zero number then $\tau_+=\tau_-$ and $C$ is a strip over a chord.
\end{itemize}
\end{lem}

\begin{proof}
We prove the first assertion. Let $C \in \M(J,\gamma_+,\gamma_-)$ be fixed by translation by $c\neq0$. Represent $C$ by a finite-energy cylinder $\util(s,t) = (a(s,t),u(s,t))$. We find $s_0\in \R$ and $t_0\in\R/\Z$ such that $a(s,t)+c=a(s+s_0,t+t_0)$ and $u(s,t)=u(s+s_0,t+t_0)$ holds for all $(s,t)$. First note that $s_0 \neq 0$, in fact if $s_0=0$ then for any $s\in\R$ $$ c + \int_0^1a(s,t)dt = \int_0^1 c + a(s,t) \ dt = \int_0^1 a(s,t+t_0)dt = \int_0^1a(s,t)dt $$ forcing $c$ to be zero, absurd. We only handle the case $s_0>0$, the case $s_0<0$ is analogous. For any $k\in\Z$ we have $u(s+ks_0,t+kt_0)=u(s,t)$. Let $k_j \to +\infty$ be a sequence such that $k_jt_0 \to t_+ \in\R/\Z$ as $j\to+\infty$. Taking the limit as $j\to+\infty$ in  the identity $u(s,t)=u(s+k_js_0,t+k_jt_0)$ we conclude that $u(s,t) = \gamma_+(T_+(t+t_+)+d)$ for some $d\in\R$, where we write $\gamma_+=(\gamma_+,T_+)$. Arguing analogously we conclude that $u(s,t) = \gamma_-(T_-(t+t_-)+d')$ for some $t_-,d'$, where $\gamma_-=(\gamma_-,T_-)$. Hence $\gamma_+=\gamma_-$ and $\util$ is a cylinder over a periodic orbit. The argument for strips is analogous and left to the reader.
\end{proof}

Now let $(\mathbb{R} \times Y,\varpi)$ be an exact symplectic cobordism from $\lambda^+$  to  $\lambda^-$ as in Section~\ref{sec_exact}. Let $J^+ \in \mathcal{J}(\lambda^+)$, $J^- \in \mathcal{J}(\lambda^-)$. Consider an almost complex structure $\bar J \in \mathcal{J}_\varpi(J^+,J^-)$ and two Reeb chords $\tau_+ \in  \T(\lambda^+,\Lambda_0 \to \Lambda_1)$, $\tau_- \in  \T(\lambda^-,\Lambda_0 \to \Lambda_1)$. We denote by $\mathcal{M}(\bar J,\tau_+,\tau_-)$
the space of equivalence classes of finite-energy pseudoholomorphic strips $\wtil: \R\times[0,1] \to (\mathbb{R}\times Y,\bar J)$ such that
\begin{itemize}
\item $\wtil$ is positively asymptotic to $\tau_+$ at $+\infty \times [0,1]$;
\item $\wtil$ is negatively asymptotic to $\tau_-$ at $-\infty \times[0,1]$;
\item $\wtil$ satisfies~\eqref{boundary_condition_strips}.
\end{itemize}
As above, two strips $\wtil(s,t)$ and $\wtil'(s,t)$ represent the same element in $\mathcal{M}(\bar J,\tau_+,\tau_-)$ if and only if $\wtil'(s,t) = \wtil(s+s_0,t)$ for some $s_0\in \R$.

For given $\gamma_+\in\P(\lambda^+)$ and $\gamma_-\in\P(\lambda^-)$ one may define moduli spaces $\M(\bar J,\gamma^+,\gamma^-)$ of finite-energy cylinders in exact cobordisms just as above, details are omitted.

\subsection{Conley-Zehnder index and Fredholm index} \label{section:index}

Let $(Y,\xi)$ be a contact $3$-manifold and $\Lambda_0$ and $\Lambda_1$ be disjoint Legendrian knots in $(Y,\xi)$. Choose orientations for $\Lambda_0$ and $\Lambda_1$. 

\subsubsection{The Conley-Zehnder index}

Let $\lambda$ be a contact form on $(Y,\xi)$ and let the Reeb chord $\tau\in\T(\lambda,\Lambda_0\to\Lambda_1)$ have length $\ell$. To define the Conley-Zehnder index of $\tau$ we first choose a non-vanishing section~$\eta$~of~$\tau^*\xi$ such that:
\begin{itemize}
\item $\eta(0)$ is tangent to $T_{\tau(0)} \Lambda_0$ and matches the orientation that we chose for $\Lambda_0$. 
\item $\eta(\ell)$ is tangent to $T_{\tau(\ell)}\Lambda_1$ and matches the orientation that we chose for $\Lambda_1$. 
\end{itemize}
Using $\eta$ we can define a symplectic trivialization $\Psi_\tau: [0,\ell] \times \mathbb{R}^2 \to \tau^*\xi$ which satisfies $\Psi_\tau(\{t\} \times \mathbb{R})= \mathbb{R}\eta(t)$ for every $t \in [0,\ell]$. Transporting the Lagrangian subspace $T_{\tau(0)} \Lambda_0$ by the Reeb flow $\phi_\lambda$ we obtain the path $\{(D\phi^t_\lambda)_{\tau(0)} T_{\tau(0)} \Lambda \mid t \in [0,\ell]\}$. Using the symplectic trivialisation $\Psi_\tau$ we can represent this path as a path $\{L_{\Psi_\tau}(t) \mid t\in [0,\ell] \}$ of Lagrangian subspaces of $(\mathbb{R}^2,\omega_0)$. The path $L_{\Psi_\tau}$ is not closed. In order to make it closed we concatenate it with a path that turns $L_{\Psi_\tau}(\ell)$ to the left until it first reaches $\mathbb{R}\times\{0\}$.

\begin{defn}
The Conley Zehnder index $\mu^{\eta}_{\cz}(\tau) \in \mathbb{Z}$ is the Maslov index of the closed path $L_{\Psi_\tau}$ of Lagrangian subspaces in $(\mathbb{R}^2,\omega_0)$. The $\mathbb{Z}_2$-degree $|\tau|$ of the Reeb chord $\tau$ is defined as the parity of  $\mu^\eta_{\cz}(\tau)$, which does not depend on $\eta$. We will say that a Reeb chord is even if $|\tau|$ is even and odd if $|\tau|$ is odd.
\end{defn}

If $\gamma=(\gamma,T) \in \P(\lambda)$ and $\Psi$ is a $d\lambda$-symplectic trivialization of $\gamma(T\cdot)^*\xi \to \R/\Z$ then we can identify the linearized flow $\{(D\phi_\lambda^t)_{\gamma(0)}\mid t\in[0,T]\}$ with a path of symplectic $2\times 2$ matrices $M(t)$.

\begin{defn}
The Conley-Zehnder index of $\gamma$ with respect to $\Psi$ is defined to be the Conley-Zehnder index of the path $t\in[0,T] \mapsto M(t)$. It will be denoted by $\mu_{\cz}^\Psi(\gamma) \in \Z$.
\end{defn}

\begin{rem}
We do not restrict to non-degenerate periodic Reeb orbits to define the Conley-Zehnder index. If $\gamma$ is degenerate then we take the unique lower semi-continuous extension of the Conley-Zehnder index to degenerate symplectic paths.
\end{rem}

\subsubsection*{The special case of unit tangent bundles}

We specialize our discussion to the case where the contact $3$-manifold is the unit tangent bundle $(T^1S,\xi_{\mathrm{geo}})$ of a closed oriented surface $S$ endowed with the geodesic contact structure $\xi_{\mathrm{geo}}$, and $\Lambda_0$ and $\Lambda_1$ are the unit tangent fibers over distinct points in $S$. Note that $\Lambda_0,\Lambda_1$ are naturally oriented by $S$.

We choose a nowhere vanishing vector field in $(T^1S,\xi_{\mathrm{geo}})$ that is tangent to the unit tangent fibers. We use this vector field and the orientation of $S$ to generate a global trivialization of $\xi$ that we denote by $\Psi$. We perform the same construction as before, using the trivialization $\Psi$, to obtain for each Reeb chord $\tau$ the path of Lagrangian subspaces $L_{\tau}$ in $(\mathbb{R}^2,\omega_0)$.

\begin{defn}
In this setting, the integral Conley-Zehnder index $\widetilde{\mu}_{\cz}(\tau) \in \mathbb{Z}$ of a Reeb chord $\tau$ is the Maslov index of $L_{\tau}$. We define the $\mathbb{Z}$-grading $|\tau|_{\mathbb{Z}}$ of $\tau$ to be $\widetilde{\mu}_{\cz}(\tau)$.
\end{defn}

\begin{defn}
In this setting, the Conley-Zehnder index $\widetilde{\mu}_{\cz}(\gamma) \in \mathbb{Z}$ of a Reeb orbit $\gamma \in \P(\lambda)$ is the Conley-Zehnder index $\mu_{\cz}^\Psi(\gamma)$ of the pair $\gamma,\Psi$.
\end{defn}

\subsubsection{The Fredholm index}

Consider $\lambda_\pm$ two defining contact forms on $(Y,\xi)$ satisfying $\lambda_+ = f\lambda_-$ for some smooth $f>1$ on $Y$. We consider an exact symplectic cobordism $(\R\times Y,\varpi)$ as in Section~\ref{sec_exact}. We also fix $J_\pm \in \J(\lambda_\pm)$ and $\bar J \in \J(J_+,J_-)$.

Fix $\tau_\pm \in \T(\lambda_\pm,\Lambda_0\to\Lambda_1)$ Reeb chords, and let $\widetilde{u}$ be a strip representing an element of $\M(\bar J,\tau_+,\tau_-)$. Assume that $\tau_\pm$ are transverse Reeb chords. In a certain functional analytic set-up the non-linear Cauchy-Riemann operator $\bar\partial_{\bar J}$ can be linearized at $\util$ giving rise to a linear Fredholm operator whose index we denote by $I(\widetilde{u})$; see~\cite{wendl_transv} for more details. Arguments from~\cite{Ab1,Ekholm} show that 
$$
I(\widetilde{u}) = \mu_{CZ}^\eta(\tau^+) - \mu_{CZ}^\eta(\tau^-),
$$
where $\eta$ is a non-vanishing section of $\util^*\xi$ that is tangent to $\Lambda_0$ over $ \R \times \{0\} $ matching the orientation we chose for $\Lambda_0$, and is tangent to  $\Lambda_1$ over $ \R \times \{1\} $ matching the orientation that we chose for $\Lambda_1$. It follows that 
$$
I(\widetilde{u}) \equiv |\tau_+| - |\tau_-| \mod 2
$$

Fix $\gamma_\pm \in \P(\lambda_\pm)$ periodic Reeb orbits, and let $\widetilde{u}=(a,u)$ be a cylinder representing an element of $\M(\bar J,\gamma_+,\gamma_-)$. Assume that $\gamma_+,\gamma_-$ are non-degenerate. In a certain functional analytic set-up the non-linear Cauchy-Riemann operator $\bar\partial_{\bar J}$ can be linearized at $\util$ giving rise to a linear Fredholm operator whose index we denote by $I(\widetilde{u})$. Standard arguments show that
$$
I(\widetilde{u}) = \mu_{\cz}^\tau(\gamma_+) - \mu_{\cz}^\tau(\gamma_-)
$$
where $\tau$ is any trivialization of $u^*\xi$ that extends to trivializations of $\xi$ over the asymptotic limits $\gamma_\pm$.

\subsubsection*{The special case of unit tangent bundles}

In the special case where the contact manifold is the unit tangent bundle $(T_1S,\xi_{\mathrm{geo}})$ of a closed oriented surface $S$ endowed with the geodesic contact structure $\xi_{\mathrm{geo}}$, and $\Lambda_0$ and $\Lambda_1$ are the unit tangent fibers over distinct points, we have that $$ I(\widetilde{u}) = |\tau_+|_{\mathbb{Z}} - |\tau_-|_{\mathbb{Z}} $$ for a finite-energy strip $\widetilde{u}$ between chords $\tau_\pm$. Similarly, $$ I(\widetilde{u}) = \widetilde\mu_{\cz}(\gamma_+) - \widetilde\mu_{\cz}(\gamma_-) $$ for a finite-energy cylinder $\widetilde{u}$ between closed Reeb orbits $\gamma_\pm$.

\begin{rem}
The above discussion and formulas apply to the case of finite-energy cylinders or strips on a symplectization $(\R\times Y,d(e^s\lambda))$.
\end{rem}

\section{Strip Legendrian contact homology on the complement of a finite set of periodic Reeb orbits}\label{section3}

In this section we show that under certain conditions one can construct a version of Legendrian contact homology on the complement of a finite set of Reeb orbits. We call this homology theory, which first appeared in the first author's PhD thesis \cite{thesis}, the strip Legendrian contact homology on the complement of a set of Reeb orbits.

The goal behind the construction of this theory is to study forcing of Reeb chords. The idea of using SFT invariants to study forcing of Reeb orbits comes from the work of Momin \cite{Momin} and was further developed in \cite{HMS}. It is briefly described as follows: on a contact $3$-manifold $(Y,\xi)$ we consider a transverse link, which we denote by $L$, and we study dynamical properties of contact forms on $(Y,\xi)$ which have the link $L$ as a set of Reeb orbits. We address the question whether $L$ forces infinitely many periodic orbits or even positive topological entropy.

The works \cite{HMS,Momin} have answered this question positively proving that certain transverse links force the existence of infinitely many Reeb orbits. In these works, the cylindrical contact homology on the complement of $L$ is used to study how the condition of having $L$ as a set of Reeb orbits forces the existence of other Reeb orbits. We follow a similar approach focusing on the implications of certain transverse links to the existence of Reeb chords and consequently to the complexity of the dynamics.

We start presenting the setup on which we work: let $L$ be a transverse link in the contact $3$-manifold $(Y,\xi)$ and $\lambda_0$ be a contact form on $(Y,\xi)$ for which the link $L$ consists of Reeb orbits of $X_{\lambda_0}$. Let $\Lambda$ and  $\widehat \Lambda$ be a pair of disjoint compact Legendrian submanifolds which do not intersect $L$. Denote by $\pi_1(Y\setminus L, \Lambda)$  the set of homotopy classes of curves in $Y \setminus L$  which start and end at $\Lambda$.

Denote by $\pi_1(Y\setminus L, \Lambda \to \widehat \Lambda)$  the set of homotopy classes of curves in $Y \setminus L$  which start at $\Lambda$ and  end at $\widehat \Lambda$. Fixing $\rho \in \pi_1(Y \setminus L,\Lambda \to \widehat \Lambda)$ and $C>0$ we denote by $$\mathcal{T}^{\rho,C}(\lambda_0)\subset \mathcal{T}(\lambda_0,\Lambda \to \widehat \Lambda)$$ the set of Reeb chords of $X_{\lambda_0}$  from $\Lambda$ to $\widehat \Lambda$ in the homotopy class $\rho$ and which have action $\leq C$.

We shall assume the following hypotheses on $\lambda_0$:
\begin{itemize}
\item[(a)] every Reeb orbit in $Y\setminus L$  is non-contractible in $Y\setminus L$,
\item[(b)] every Reeb orbit $\gamma_{L}$ in $L$ is either non-contractible or for any disc $D_{\gamma_L}\subset Y$ with $\partial D_{\gamma_L}=\gamma_L$, the interior of $D_{\gamma_L}$ intersects $L$,
\item[(c)] every Reeb chord from $\Lambda$ to itself does not vanish in $\pi_1(Y \setminus L,\Lambda)$,
\item[(d)] every Reeb chord from $\widehat \Lambda$ to itself  does not vanish in $\pi_1(Y \setminus L,\widehat \Lambda)$.
\end{itemize}
If it satisfies conditions (a)-(d) then $\lambda_0$ is said to be adapted to $(Y\setminus L, \Lambda \to \widehat \Lambda)$.

In order to have a well-defined Legendrian contact homology on $Y\setminus L$, filtered by $\rho$ and $C$, we also assume that
\begin{itemize}
\item[(e)] every Reeb chord  in $\mathcal{T}^{\rho,C}(\lambda_0)$ is transverse and is embedded as a map from $[0,1]$ to $Y\setminus L$.
\end{itemize}

In view of the transversality condition (e) the set $\mathcal{T}^{\rho,C}(\lambda_0)$ is finite. Hence we define $${\rm LCC}^{\rho,C}(\lambda_0)=\bigoplus_{\tau \in \mathcal{T}^{\rho,S}(\lambda_0)} \Z_2 [\tau]$$ as the finite dimensional $\mathbb{Z}_2$-vector space freely generated by the elements in $\mathcal{T}^{\rho,C}(\lambda_0)$.

It is enough for our purposes to consider a $\mathbb{Z}_2$-grading on ${\rm LCC}^{\rho,C}(\lambda_0)$ which is defined on the generators using the parity of the Maslov index defined in Section \ref{section:index} and extending it linearly. This induces the splitting $${\rm LCC}^{\rho,C}(\lambda_0)={\rm LCC}^{\rho,C}_{\rm even}(\lambda_0) \oplus {\rm LCC}^{\rho,C}_{\rm odd}(\lambda_0).$$ In case there exists a well-defined $\Z$-grading for the Reeb chords from $\Lambda$ to $\widehat \Lambda$, we consider the splitting $${\rm LCC}^{\rho,C}(\lambda_0)=\bigoplus_{k \in \Z}{\rm LCC}^{\rho,C}_k(\lambda_0).$$

In order to define a differential on ${\rm LCC}^{\rho,C}(\lambda_0)$ we have to consider a special class of moduli spaces: for $J\in  \mathcal{J}(\lambda_0)$, we first remark that the set $\mathbb{R}\times  L$ is the union of a finite number of $J$-holomorphic cylinders in the symplectization $(\mathbb{R}\times Y, J)$.

For $\tau,\tau' \in \mathcal{T}^{\rho,C}(\lambda_0)$, we define  $$\mathcal{M}_k^{\rho,C} (J,\tau,\tau')\subset \mathcal{M}_k(J,\tau,\tau')$$ as the subset of $\mathcal{M}_k(J,\tau,\tau')$ consisted of  $J$-holomorphic strips whose image in $\R \times Y$ do not intersect $\R \times  L$. Analogously, we define $\widetilde{\mathcal{M}}_k^{\rho,C} (J,\tau,\tau')$ as the quotient of $\mathcal{M}_k^{\rho,C} (J,\tau,\tau')$ by the $\R$-action that comes from the fact that $J$ is $\R$-invariant. %We let $\overline{\mathcal{M}}_k^{\rho,S} (\tau,\tau',J,\Lambda,\widehat \Lambda)$ be the SFT-compactification of  $\widetilde{\mathcal{M}}_k^{\rho,S} (\tau,\tau',J,\Lambda,\widehat \Lambda)$.

It follows from condition (e) above we can apply the asymptotic analysis of \cite{Ab} to conclude that for any pair $\tau,\tau' \in \mathcal{T}^{\rho,C}(\lambda_0)$ the moduli space $\mathcal{M}_k(J,\tau,\tau')$ contains only somewhere injective pseudoholomorphic strips. In \cite[Proposition 3.15]{Rizzel} Dimitroglou Rizzel combined the techniques of \cite{Ab1}, \cite{Dr} and \cite{Lazzarini} to show that  for a generic set $\mathcal{J}^\rho_{\rm reg}(\lambda_0) \subset \mathcal{J}(\lambda_0)$, we have that for  $J\in \mathcal{J}^\rho_{\rm reg}(\lambda_0)$ every element in the moduli space $\mathcal{M}^{\rho,C}_k(J,\tau,\tau')$ is transverse in the sense of Fredholm theory, i.e., the linearization of the Cauchy-Riemann operator $\bar \partial_J$ along any element of $\mathcal{M}^{\rho,C}_k(J,\tau,\tau')$ is surjective, this being valid for all $\tau,\tau' \in \mathcal{T}^{\rho,C}(\lambda_0)$. It then follows that if $J\in  \mathcal{J}^\rho_{\rm reg}(\lambda_0)$ then  $$\widetilde{\mathcal{M}}_k^{\rho,C} (J,\tau,\tau') \mbox{ is a } (k-1)-\mbox{dimensional manifold},$$ for any $\tau,\tau'$ and $k\geq 1$. This follows from the fact that \begin{equation}\label{interL}{\rm closure}({\rm image}(w)) \cap L = \emptyset,\end{equation} for any $\wtil=(a,w)\in \mathcal{M}_k^{\rho,C} (J,\tau,\tau')$ and hence \eqref{interL} holds true also for all elements in a small neighbourhood of $\wtil$ in $\mathcal{M}_k (J,\tau,\tau')$.

We are ready to define a differential on ${\rm LCC}^{\rho,C}(\lambda_0)$.
\begin{defn}
Let $\tau \in \mathcal{T}^{\rho,C}(\lambda_0)$ and $J \in \mathcal{J}^\rho_{\rm reg}(\lambda_0) \subset \mathcal{J}(\lambda_0)$. We define
\begin{equation}
d^{\rho}_J (\tau)= \sum_{\tau' \in \mathcal{T}^{\rho,C}(\lambda_0)}  n(\tau,\tau')_{\Z_2}  \tau'
\end{equation}
where $n(\tau,\tau')$ is the (finite) number of elements in $\widetilde{\mathcal{M}}_1^{\rho,C} (J,\tau,\tau')$ of $J$-holomorphic strips of Fredholm index $1$ modulo the $\mathbb{R}$-action.
\end{defn}

We must now prove that the differential $d^\rho_J$ is well-defined.
\begin{lem}\label{lemmadifferential}
Assume that $\lambda_0$ is adapted to the pair $(Y\setminus L, \Lambda \to \widehat \Lambda)$. Fix $\rho \in \pi_1(Y\setminus L, \Lambda \to \widehat \Lambda)$, $C>0$ and assume that condition (e) holds. Then for any $J \in \mathcal{J}_{\rm reg}^\rho(\lambda_0) \subset \mathcal{J}(\lambda_0)$ the following assertions hold:
\begin{enumerate}
\item For each $\tau \in \mathcal{T}^{\rho,C}_{\Lambda \to \widehat \Lambda}(\lambda_0)$, $d^\rho_J(\tau)$ is a finite sum and we can extend $d^\rho_J$ to a well defined linear map from ${\rm LCC}^{\rho,C}(\lambda_0)$ to itself.
\item The differential $d^{\rho}_J$ decreases the action of Reeb chords.
\item We have $$\begin{aligned} d ^\rho_J({\rm LCC}^{\rho,C}_{\rm even}(\lambda_0)) \subset  {\rm LCC}^{\rho,C}_{\rm odd}(\lambda_0),\\ d^\rho_J({\rm LCC}^{\rho,C}_{\rm odd}(\lambda_0)) \subset  {\rm LCC}^{\rho,C}_{\rm even}(\lambda_0).\end{aligned}$$ In case the Reeb chords in class $\rho$ have a well defined integer Conley-Zehnder index, we have $$d_J^\rho({\rm LCC}^{\rho,C}_{k}(\lambda_0)) \subset  {\rm LCC}^{\rho,C}_{k-1}(\lambda_0).$$
\end{enumerate}
\end{lem}

\proof We first prove that $d^\rho_J$ is well-defined. For that, we show that  $\widetilde{\mathcal{M}}(J,\tau,\tau')$ is finite for every $\tau$ and $\tau'$ in  $\mathcal{T}^{\rho,C},(\lambda_0)$. Since $J \in \mathcal{J}_{\rm reg}^\rho(\lambda_0)$, $\widetilde{\mathcal{M}}_1^{\rho,C} (J,\tau,\tau')$ is a $0$-dimensional manifold. We show that it is compact which implies that it has to be a finite set.

To obtain the compactness of  $\widetilde{\mathcal{M}}_1^{\rho,C}(J,\tau,\tau')$, we apply the standard bubbling-off analysis for pseudoholomorphic curves of \cite{H} and the SFT compactness results of \cite{CPT}.

Let $\tilde w_n$ be a sequence of elements in $\widetilde{\mathcal{M}}_1^{\rho,C}(J,\tau,\tau')$. By assumption the energy of $\tilde w_n$ is uniformly bounded. We first argue that there are no bubbling-off points.

Suppose that there exists an interior bubbling-off point for the sequence $\tilde w_n$. Then a subsequence of $\tilde w_n$ converges  to a pseudoholomorphic building $\tilde w$ which has a pseudoholomorphic plane $\tilde w_{\rm pl}$ as one of its components. If $\tilde w_{\rm pl}$ intersected $\R \times L$ then these intersections would be isolated and by positivity and stability of intersections, it would follow that  $\tilde w_n$ also intersects $\R \times L$, which is a contradiction. Hence $\tilde w_{\rm pl}$ does not intersect $\R \times L$. Let $\gamma$ be an asymptotic limit of $\tilde w_{\rm pl}$. There are two possibilities:
\begin{itemize}
    \item if $\gamma$ is in $L$ then hypothesis (b) implies that $\tilde w_{\rm pl}$ intersects $\R \times L$, which is impossible;
    \item if $\gamma$ is in $Y \setminus L$, then hypothesis (a) implies again implies that $\tilde w_{\rm pl}$ intersects $\R \times L$, which is also impossible.
\end{itemize}
 We thus conclude that $\tilde w_n$ has no interior bubbling-off point.

To see that the sequence $\tilde w_n$ has no boundary bubbling-off point we proceed similarly. Assuming that such a bubbling-off point exists we know that there exists a subsequence of $\tilde w_n$ converging to a pseudoholomorphic building $\tilde w$ such that at least one component of $\tilde w$ is a pseudoholomorphic half-plane $\tilde w_{\rm hp}$. In this case $\tilde w_{\rm hp}$ cannot intersect $\R \times L$ since any such intersection would be isolated and by positivity and stability of intersections $\tilde w_n$ would also intersect $\R \times L$, a contradiction. The boundary of $\tilde w_{\rm hp}$ is entirely contained in either $\mathbb{R}\times  \Lambda$ or $\mathbb{R}\times  \widehat \Lambda$ and its only positive puncture is asymptotic to a Reeb chord $\tau_{\rm hp}$ going either from $\Lambda$ to itself or from  $ \widehat \Lambda$ to itself. Since $\tilde w_{\rm hp}$ does not intersect $\R \times L$, we conclude that $\tau_{\rm hp}$ is trivial in $\pi_1(Y \setminus L,\Lambda)$ or in $\pi_1(Y \setminus L,\widehat{\Lambda})$, something that contradicts hypotheses (b) or (c), respectively. We thus conclude that $\tilde w_n$ has no boundary bubbling-off point.

Since there are no bubbling-off points for the sequence $\tilde w_n$, the SFT-compactness theorem tells us that there is a subsequence of $\tilde w_n$ which converges in the SFT sense to a pseudoholomorphic building $\tilde w$ with $q$-levels $\tilde w^l$ (where $l=1\ldots q$) all of which are pseudoholomorphic strips. Each $\tilde w^l: \D \setminus \{-1,1\} \to  \mathbb{R}\times Y$ satisfies:
\begin{itemize}
\item $1$ is a positive boundary puncture and $\tilde w^l$ is asymptotic to $\tau_l \in \mathcal{T}^{\rho,C}(\lambda_0)$ at $1$;
\item $-1$ is a negative boundary puncture and $\tilde w^l$ is asymptotic to $\tau_{l+1} \in \mathcal{T}^{\rho,C}(\lambda_0)$ at $-1$;
\item $\tilde w^l(H_-) \subset \mathbb{R} \times \Lambda$;
\item $\tilde w^l(H_+) \subset \mathbb{R} \times \widehat \Lambda$,
\end{itemize}
where $\tau_1 = \tau$, $\tau_{q+1} = \tau'$ and $\tau_l \neq \tau_{l+1} $.

By the regularity of $J$ and the fact that every $\tilde w^l$ is a somewhere injective pseudoholomorphic strip different from a trivial strip over a Reeb chord, we know that the Fredholm index $I_F(\tilde w^l)$ is at least $1$. Since the Fredholm index of the building $\tilde w$ is the sum of the Fredholm indices of its levels we have $I_F (\tilde w)= \sum_{l=1}^{q} I_F(\tilde w^l) \geq q.$ On the other hand, since $\tilde w$ is the SFT-limit of a sequence of pseudoholomorphic strips with Fredholm index 1, we must have $I_F(\tilde w)=1$. As a consequence we have $q=1$ and $\tilde w\in \widetilde{\mathcal{M}}_1(J,\tau,\tau')$. To see that $\tilde w$ is actually an element of $\widetilde{\mathcal{M}}_1^{\rho,C} (J,\tau,\tau')$ we argue indirectly. If this is not the case then there is an isolated intersection of $\tilde w$ with $\mathbb{R}\times  L$. Again the positivity and stability of intersections imply that $\tilde w_n$ intersect $\R \times L$, a contradiction.

We have shown that every sequence of elements in $\widetilde{\mathcal{M}}^{\rho,C}_1 (J,\tau,\tau')$ has a subsequence that converges to an element in  $\widetilde{\mathcal{M}}^{\rho,C}_1 (J,\tau,\tau')$. This implies the desired compactness of  $\widetilde{\mathcal{M}}^{\rho,C}_1 (J,\tau,\tau')$. Since $\widetilde{\mathcal{M}}^{\rho,C}_1 (J,\tau,\tau')$ is a $0$-dimensional manifold, it follows that it is a finite set and hence $n(\tau,\tau')$ is a finite number for every pair of Reeb chords $\tau$ and $\tau'$ in $\mathcal{T}^{\rho,C}(\lambda_0)$. Condition (e) implies that, for fixed $\tau$, the number $n(\tau,\tau')$ is non-zero only for a finite number of Reeb chords $\tau'$. This proves (1).

Notice that $n(\tau,\tau')$ can only be non-zero for Reeb chords $\tau, \tau'$ with $A(\tau') < A(\tau)\leq C$. Indeed, if $A(\tau')\geq A(\tau)$, then a pseudoholomorphic strip $\tilde w$ which is positively asymptotic to $\tau$ at $1$ and negatively asymptotic to $\tau'$ at $-1$ exist only if $\tau=\tau' \Rightarrow A(\tau)=A(\tau')$ and $\tilde w$ is a cylinder over $\tau$ . In this case, the Fredholm index of $\tilde w$ is $0$ and hence $\mathcal{M}_1^{\rho,C}(J,\tau,\tau')=\emptyset.$ This proves (2).

Item (3) follows easily from the fact that the Fredholm index of a pseudoholomorphic strip with one positive puncture asymptotic to a Reeb chord $\tau \in \mathcal{T}^{\rho,C}(\lambda_0)$ and one negative puncture asymptotic to a Reeb chord $\tau' \in \mathcal{T}^{\rho,C}(\lambda_0)$ has the same parity as the sum of the parities of $\tau$ and $\tau'$. Thus $\widetilde{\mathcal{M}}^{\rho,C}_1 (J,\tau,\tau')$ can be non-empty only if $\tau$ and $\tau'$ have different parity, from which (3) follows immediately. \qed

Having established these properties of  $d^{\rho}_J$ we proceed to show that $d^{\rho}_J\circ d^{\rho}_J = 0$. Again, this is only true since $\lambda_0, \rho$ and $C$ satisfy conditions (a)-(e) above.
\begin{lem} Assume that $\lambda_0$ is adapted to the pair $(Y \setminus L, \Lambda \to \widehat \Lambda)$ and condition (e) holds for $\rho,C$. If $J \in \mathcal{J}_{\rm reg}^\rho(\lambda_0) \subset \mathcal{J}(\lambda_0)$, then
\begin{equation}
d^{\rho}_J \circ d^{\rho}_J = 0.
\end{equation}
\end{lem}

\proof The regularity of $J$ implies that each connected component $I$ of $\widetilde{\mathcal{M}}^{\rho,C}_2 (J,\tau,\tau')$ is a 1-dimensional manifold. We assume $I$ is an open interval and let $\tilde w_n\in I$ be a sequence converging to one of its boundary points. Reasoning as in the proof of Lemma \ref{lemmadifferential} one obtains that no sequence of bubbling-off points occurs. Thus the SFT compactness theorem, together with the stability and positivity of intersections,  imply that some subsequence of $\tilde w_n$ converges to a pseudoholomorphic building $\tilde w$ with $q$-levels $\tilde w^l$, where $1 \leq l \leq q $, and each level consists of a somewhere injective finite energy strip $\tilde w^l\in \widetilde{\mathcal{M}}_{k_l}^{\rho,C}(J,\tau_{l},\tau_{l+1})$ for some  $k_l\geq 1$. The Reeb chords $\tau_l,  \in\mathcal{T}^{\rho,C}(\lambda_0)$ for $l=1\ldots q+1$,
%\begin{itemize}
%\item $1$ is a positive boundary puncture where $\tilde w^l$ is asymptotic to $\tau_l \in\mathcal{T}^{\rho,S}(\lambda_0)$.
%\item $-1$ is a negative boundary puncture where $\tilde w^l$ is asymptotic to $\tau_{l+1} \in \mathcal{T}^{\rho,S}(\lambda_0)$.
%\item %$\tilde w^l(H_-) \subset \mathbb{R} \times \Lambda$.
%\item $\tilde w^l(H_+) \subset \mathbb{R} \times \widehat \Lambda$.
%\end{itemize}
with $\tau_1 = \tau$ and $\tau_{q+1} = \tau'$.

Since $\tilde w$ is the SFT limit of a sequence of pseudoholomorphic strips of Fredholm index 2 we have $2=I_F(\widetilde{w})=\sum_{l=1}^q k_l \geq q$. If $q=1$ then $\tilde w \in \widetilde{\mathcal{M}}^{\rho,S}_2(J, \tau,\tau')$. However, this case is ruled out since we are assuming that the sequence $\tilde w_n$ is converging to one of the boundary points of $I$. Hence $q=2$ which forces $k_1=k_2=1$, $\tilde w^1 \in \widetilde{\mathcal{M}}_1^{\rho,C} (J,\tau,\tau_2)$ and $\tilde w^2 \in \widetilde{\mathcal{M}}^{\rho,C}_1 (J,\tau_2,\tau')$.

We conclude that we can associate to each boundary point of $I \subset \widetilde{\mathcal{M}}_2^{\rho,C} (J,\tau,\tau')$  a $2$-level pseudoholomorphic building, whose levels are elements  in the moduli spaces $\widetilde{\mathcal{M}}^{\rho,C}_1 (J,\tau,\tau'')$ and $\widetilde{\mathcal{M}}^{\rho,C}_1 (J,\tau'',\tau')$ for some $\tau''\in \mathcal{T}^{\rho,C}(\lambda_0)$.

On the other hand the gluing theorem gives the description of a neighbourhood of such a 2-level building in the SFT compactification of $\widetilde{\mathcal{M}}^{\rho,C} (J,\tau,\tau')$, the levels being elements in $\widetilde{\mathcal{M}}^{\rho,C}_1 (J,\tau,\tau'')$ and $\widetilde{\mathcal{M}}^{\rho,C}_1 (J,\tau'',\tau')$, with $\tau,\tau',\tau''\in \mathcal{T}^{\rho,C}(\lambda_0)$. This neighbourhood is diffeomorphic to the interval $[0, +\infty)$, taking $0$ to the 2-level building and all other points to elements in $\widetilde{\mathcal{M}}^{\rho,C}_2 (J,\tau,\tau')$.

Summing up, the compactification of $\widetilde{\mathcal{M}}^{\rho,C}_2 (J,\tau,\tau')$ has the structure of a $1$-dimensional manifold with boundary and consists of the disjoint union of finitely many circles and compact intervals.

Now observe that $d^{\rho}_J \circ d^{\rho}_J (\tau)$ has the form  $\sum_{\tau' \in \mathcal{T}^{\rho,S}(\lambda_0)} m(\tau,\tau')_{\Z_2}\tau'$, where $$m(\tau,\tau'):=\sum_{\tau''\in \mathcal{T}^{\rho,S}(\lambda_0)}n(\tau,\tau'')n(\tau'',\tau'),$$ and $n(\tau,\tau')$ is the number of elements in $\widetilde{\mathcal{M}}^{\rho,C}_1(J,\tau,\tau')$.

From the description of the compactification of $\widetilde{\mathcal{M}}_2^{\rho,C}(J,\tau,\tau')$ above, we conclude that the number $m(\tau,\tau')$ coincides with the number of boundary components of the compact intervals contained in the compactification of $\widetilde{\mathcal{M}}^{\rho,C}_2 (J,\tau,\tau')$, which is necessarily even. Hence $m(\tau,\tau')_{\Z_2}=0$ and the lemma follows.
\qed

We have thus obtained that under appropriate conditions the Legendrian contact homology on the complement of $L$ is well defined. We denote by
${\rm LCH}^{\rho,C}(J)= \frac{\ker d^\rho_J}{{\rm image} d^\rho_J}$ the homology associated to the differential chain complex $({\rm LCC}^{\rho,C}(\lambda_0),d^{\rho}_J)$. We have  the splitting $${\rm LCH}^{\rho,C}(J)= {\rm LCH}^{\rho,C}_{\rm even}(J) \oplus {\rm LCH}^{\rho,C}_{\rm odd}(J)$$ or, in case the Reeb chords admit a $\Z$-grading, $${\rm LCH}^{\rho,C}(J)=\bigoplus_{k\in \Z} {\rm LCH}^{\rho,C}_k(J).$$

Next we proceed to construct cobordism maps for the Legendrian contact homology.

\subsection{Chain maps} Here we consider the following situation. Let $(\mathbb{R} \times Y,d\kappa)$ be an exact symplectic cobordism from  $\lambda_0$ to $c\lambda_0$, where $0<c<1$ is a constant. Let $\mathcal{L}:=\R \times \Lambda$ and $\widehat{\mathcal{L}}:= \R \times \widehat{\Lambda}$. Then, $\mathcal{L}$ and $\widehat{\mathcal{L}}$  are exact Lagrangian cobordisms from $\Lambda$ to itself and from $\widehat \Lambda$ to itself, respectively. Assume that $\mathbb{R} \times L$ is a symplectic submanifold of $(\mathbb{R} \times Y,d\kappa)$. Following \cite{HMS,Momin} we consider the space $\mathcal{J}^L(J_+,J_-)$ of almost complex structures on $\mathbb{R} \times Y$ which are compatible with $d\kappa$,  coincide with $J_+\in \mathcal{J}(\lambda_0)$ on $[1,+\infty) \times Y$ and coincide with $J_-\in \mathcal{J}(c\lambda_0)$ on $(-\infty,-1]\times Y$. We also require that the set  $\R \times L$ is a union of pseudoholomorphic cylinders.

Given $\bar J \in \mathcal{J}^L(J_+,J_-)$, $\tau \in \mathcal{T}^{\rho,C}(\lambda_0)$ and $\tau' \in \mathcal{T}^{\rho,C}(c\lambda_0)$, we consider the moduli spaces $\mathcal{M}^{\rho,C}_k(\bar J, \tau,\tau')$ of pseudoholomorphic strips which do not intersect $\mathbb{R}\times L$, have boundary in $\mathcal{L} \cup \widehat{\mathcal{L}}$ and have Fredholm index $k$, as considered in section \ref{pseudohol}. Under condition (e) we know that for all $\tau,\tau'$ the elements in $\mathcal{M}^{\rho,C}_k(\bar J, \tau,\tau')$ are somewhere injective. It then follows from perturbation techniques used in \cite{HMS,Momin} combined with the techniques used in the proof of \cite[Proposition 3.15]{Rizzel} that there is a generic set $\mathcal{J}_{\rm reg}^{L,\rho}(J_+,J_-) \subset  \mathcal{J}^{L}(J_+,J_-)$ such that for all $\bar J \in \mathcal{J}_{\rm reg}^{L,\rho}(J_+,J_-)$ the moduli spaces $\mathcal{M}^{\rho,C}_k(\bar J,\tau,\tau')$ are Fredholm regular for every $\tau,\tau'$. Since such curves stay apart from $\R \times L$ the same is true for nearby curves and hence $$\mathcal{M}^{\rho,C}_k(\bar J,\tau,\tau') \mbox{ is a }k\mbox{-dimensional manifold,}$$ for any $\tau,\tau'$ and $k\geq 0$.

Taking $\bar J \in \mathcal{J}_{\rm reg}^{L,\rho}(J_+,J_-)$ we define the map
\begin{equation}
\Phi_{\bar J} : {\rm LCC}^{\rho,C}(\lambda_0) \to {\rm LCC}^{\rho,C}(c\lambda_0),
\end{equation}
which on generators $\tau \in \mathcal{T}^{\rho,C}(\lambda_0)$ is given by the (finite) sum
\begin{equation}\label{defchainmap}
\Phi_{\bar J}(\tau) = \sum_{\tau' \in \mathcal{T}^{\rho,C}(c\lambda_0)}  \#_{\Z_2} \mathcal{M}^{\rho,C}_0 (\bar J,\tau,\tau')\tau',
\end{equation} where $\#_{\Z_2}$ means the cardinality $\mbox{mod }2$.

To show that $\Phi_{\bar J}$ is well defined we need to prove that $\mathcal{M}^{\rho,C}_0 (\bar J,\tau,\tau')$ is a finite set for all $\tau,\tau'$  and that for a fixed $\tau \in \mathcal{T}^{\rho,C}(\lambda_0)$ there exist finitely many chords $\tau'\in \mathcal{T}^{\rho,C}(c\lambda_0)$ for which $\mathcal{M}^{\rho,C}_0(\bar J,\tau,\tau')$ is non-empty. The proof of these two facts is completely analogous to the proof  of Lemma \ref{lemmadifferential}-(1).

\begin{prop} \label{prophomiswelldefined} Assume that $\lambda_0$ is adapted to the pair $(Y\setminus L, \Lambda \to \widehat \Lambda)$, condition (e) is satisfied, $J_+\in \mathcal{J}^\rho_{\rm reg}(\lambda_0)$, $J_-\in \mathcal{J}^{\rho}_{\rm reg}(c\lambda_0)$ and $\bar J \in \mathcal{J}_{\rm reg}^{L,\rho}(J_+,J_-)$.
Then the map $\Phi_{\bar J}$ is well defined and satisfies \begin{equation}\label{chainmap} d^{\rho}_J \circ  \Phi_{\bar J}= \Phi_{\bar J} \circ d^{\rho}_J.\end{equation} In particular, $\Phi_{\bar J}$ descends to a map on the homology level \begin{equation}\label{defichainhomology}\bar \Phi_{\bar J} :  {\rm LCH}^{\rho,C}(J_+)  \to {\rm LCH}^{\rho,C}(J_-)
%:  [\tau]  \mapsto \left[\Phi_{\bar J} \right].
\end{equation}
\end{prop}
\proof First we show that $\Phi_{\bar J}$ is well defined. Fix $\tau \in \mathcal{T}^{\rho,C}(\lambda_0)$ and $\tau' \in \mathcal{T}^{\rho,C}(c\lambda_0)$. We claim that $\mathcal{M}^{\rho,C}_0(\bar J,\tau, \tau')$ is finite. Indeed, take a sequence $\tilde w_n \in \mathcal{M}^{\rho,C}_0(\bar J,\tau, \tau')$. Arguing as in the proof of Lemma \ref{lemmadifferential} we conclude that no sequence of bubbling-off points exists. Indeed, if this is not the case then up to extraction of a subsequence the SFT limit of $\tilde w_n$ contains a pseudoholomorphic plane or a pseudoholomorphic half-plane ($\bar J,$ $J_+$ or $J_-$-holomorphic) which does not intersect $\R \times L$ by positivity and stability of intersections. The plane is asymptotic to a contractible Reeb orbit and the half-plane is asymptotic to a contractible Reeb chord in $Y \setminus L$ either from $\Lambda$ to itself or from $\widehat \Lambda$ to itself. This contradicts conditions (a)-(d)  satisfied by $\lambda_0$ since it is adapted to $(Y\setminus L, \Lambda \to \widehat \Lambda)$. Hence no sequence of bubbling-off points exists. Since the Fredholm index of $\tilde w_n$ is $0$, the SFT compactness theorem and stability and positivity of intersections tell us that up to a subsequence $\tilde w_n$ converges to a building $\tilde w$ with finitely many strips filtered by $\rho$ and $C$, at most one of them being $\bar J$-holomorphic, each one of them having non-negative Fredholm index and so that the total Fredholm index is $0$. Since the curves in $\tilde w$ which are $J_+$ or $J_-$-holomorphic have Fredholm index at least $1$, we conclude that $\tilde w$ is a single curve in $\mathcal{M}^{\rho,C}_0(\bar J,\tau, \tau')$. Now since $\mathcal{M}^{\rho,C}_0(\bar J,\tau, \tau')$ is a $0$-dimensional manifold $\tilde w_n$ must be eventually constant which implies that $\mathcal{M}^{\rho,C}_0(\bar J,\tau, \tau')$ is finite. Now condition (e) and the fact that $\Phi_{\bar J}$ decreases action  imply that for fixed $\tau$, there exists only finitely many $\tau'$ for which $\mathcal{M}^{\rho,C}_0(\bar J,\tau, \tau')$ is non-empty. Hence $\Phi_{\bar J}$ is well defined.

To prove identity \eqref{chainmap} we describe the SFT compactification of $\mathcal{M}^{\rho,C}_1 (\bar J,\tau,\tau')$, where $\tau\in \mathcal{T}^{\rho,C}(\lambda_0)$ and $\tau'\in \mathcal{T}^{\rho,C}(c\lambda_0)$. Let $I\subset \mathcal{M}^{\rho,C}_1 (\bar J,\tau,\tau')$ be a connected component homeomorphic to an open interval and let $\tilde w_n \in I$ be a sequence converging to one of its boundary points. Arguing as before we conclude that no sequence of bubbling-off points exists. Hence, from stability and positivity of intersections, the SFT compactness theorem and condition (e), we know that a subsequence of $\tilde w_n$ converges to a pseudoholomorphic building $\tilde w$ with $q$-levels, each level consisting of a somewhere injective finite energy strip $\tilde w^l$ so that, for some $1\leq s_0\leq q$, we have
\begin{itemize}
\item if $1\leq l\leq s_0-1$, then $\tilde w^l\in \widetilde{\mathcal{M}}_{k_l}^{\rho,C}(J_+,\tau_{l},\tau_{l+1})$ for some  $k_l\geq 1,\tau_l \in\mathcal{T}^{\rho,C}(\lambda_0)$.
\item $\tilde w^{s_0}\in \mathcal{M}_{k_{s_0}}^{\rho,C}(\bar J,\tau_{s_0},\tau_{s_0+1})$ for some  $k_{s_0}\geq 0,\tau_{s_0} \in\mathcal{T}^{\rho,C}(\lambda_0)$.
\item if $s_0+1\leq l\leq q$, then $\tilde w^l\in \widetilde{\mathcal{M}}_{k_l}^{\rho,C}(J_-,\tau_{l},\tau_{l+1})$ for some  $k_l\geq 1,\tau_l \in\mathcal{T}^{\rho,C}(c\lambda_0)$,
\end{itemize}
where $\tau_1 = \tau$ and $\tau_{q+1} = \tau'$.

Since $\tilde w$ is the SFT limit of a sequence of pseudoholomorphic strips with Fredholm index 1 we have $1=I_F(\widetilde{w})=\sum_{l=1}^{q} k_l \geq q-1$. If $q=1$ then $s_0=1$, $k_1=1$ and $\tilde w \in \mathcal{M}^{\rho,C}_1(\bar J, \tau,\tau')$. However, this case is ruled out since we are assuming that the sequence $\tilde w_n$ is converging to one of the boundary points of $I$. It follows that $q=2$. We then have two possibilities:
\begin{itemize}
    \item If $s_0=1$, then $k_1=0$, $k_2=1$ and  \begin{equation}\label{sit1} \tilde w^1 \in \mathcal{M}_0^{\rho,C} (\bar J,\tau,\tau_2)\mbox{ and }\tilde w^2 \in \widetilde{\mathcal{M}}^{\rho,C}_1 (J_-,\tau_2,\tau'), \end{equation} for some $\tau_2 \in \mathcal{T}^{\rho,C}(c\lambda_0)$.
    \item If $s_0=2$, then $k_1=1$, $k_2=0$ and  \begin{equation}\label{sit2} \tilde w^1 \in \widetilde{\mathcal{M}}^{\rho,C}_1 (J_+,\tau,\tau_2)\mbox{ and }\tilde w^2 \in \mathcal{M}_0^{\rho,C} (\bar J,\tau_2,\tau'), \end{equation} for some $\tau_2 \in \mathcal{T}^{\rho,C}(\lambda_0)$.
\end{itemize}

We conclude that we can associate to each boundary point of $I \subset \mathcal{M}_1^{\rho,C} (\bar J,\tau,\tau')$ a $2$-level pseudoholomorphic building either as in \eqref{sit1} or as in \eqref{sit2}. On the other hand the gluing theorem describes a neighbourhood of such 2-level buildings. This neighbourhood is homeomorphic to the interval $[0, +\infty)$, taking $0$ to the 2-level pseudoholomorphic building and all other points to elements in $\mathcal{M}^{\rho,C}_1 (\bar J,\tau,\tau')$.

The discussion above shows that the compactification of $\mathcal{M}^{\rho,C}_1 (\bar J,\tau,\tau')$ has the structure of a $1$-dimensional manifold with boundary and consists of the disjoint union of finitely many circles and compact intervals. In particular the number of boundary components of this $1$-dimensional manifold is even.

Now observe that $(d^{\rho}_J \circ  \Phi_{\bar J} - \Phi_{\bar J} \circ d^{\rho}_J)(\tau)$ has the form  $\sum_{\tau' \in \mathcal{T}^{\rho,C}(c\lambda_0)} l(\tau,\tau')_{\Z_2}\tau'$, where $$l(\tau,\tau'):=\sum_{\tau''\in \mathcal{T}^{\rho,C}(\lambda_0)}n(\tau,\tau'')m(\tau'',\tau')+\sum_{\tau''\in \mathcal{T}^{\rho,C}(c\lambda_0)}m(\tau,\tau'')n(\tau'',\tau'),$$ $m(\tau_1,\tau_2)=\# \mathcal{M}_0^{\rho,C}(\bar J,\tau_1,\tau_2)$ for $\tau_1 \in \mathcal{T}^{\rho,C}(\lambda_0)$ and $\tau_2 \in \mathcal{T}^{\rho,C}(c\lambda_0)$, and $n(\tau_1,\tau_2) \# \widetilde{\mathcal{M}}^{\rho,C}_1(J,\tau_1,\tau_2)$ for $\tau_1, \tau_2 \in \mathcal{T}^{\rho,C}(\lambda_0)$ or $\tau_1, \tau_2 \in \mathcal{T}^{\rho,C}(c\lambda_0)$.

The description above implies that $l(\tau,\tau')$ coincides with the number of boundary components of the compact intervals contained in the compactification of $\mathcal{M}^{\rho,C}_1 (\bar J,\tau,\tau')$. This number is even and hence $l(\tau,\tau')_{\Z_2}=0$. The lemma follows.
\qed

\subsection{Algebraic homotopy} Next we prove a result which is the main tool in the argument of our forcing results. Let $(\mathbb{R} \times Y,d\kappa_0)$ be an exact symplectic cobordism from the contact form $\lambda_0$ to $c\lambda_0$, $0<c<1$,  and let $\mathcal{L}:= \R \times \Lambda$ and $\widehat{\mathcal{L}}:= \R \times \widehat{\Lambda}$. The cylinders $\mathcal{L}$ and $\widehat{\mathcal{L}}$ are exact Lagrangian cobordisms in $(\mathbb{R} \times Y,d\kappa_0)$ from $\Lambda$ to itself and from $\widehat\Lambda$ to itself, respectively. We fix $\rho \in \pi_1(Y\setminus L,\Lambda \to \widehat \Lambda)$ and choose $J_{+}\in \mathcal{J}^\rho_{\rm reg}(\lambda_0)$. Let $\psi:\R \times Y \to \R \times Y$ be the diffeomorphism $\psi(a,x)=(a/c,x)$. Then $J_-:=\psi^* J_+$ lies in $\mathcal{J}^\rho_{\rm reg}(c \lambda_0)$ as one easily verifies. Let $\bar J_0 \in \mathcal{J}^{L,\rho}_{\rm reg}(J_+,J_-)$.

Now let $h_1:\R \to \R$ be a smooth diffeomorphism satisfying $h_1(a)=a/c$ if $a\leq -1$ and $h_1(a)=a$ if $a \geq 1$. Then the diffeomorphism $\psi_1:\R \times Y \to \R \times Y$, $\psi_1(a,x):=(h(a),x)$ satisfies $\bar J_1:=\psi_1^* J_+ \in \mathcal{J}^{L}(J_+,J_-)$.  Moreover $\mathcal{L}$ and $\widehat{\mathcal{L}}$ are invariant by $\psi_1$, and since $J_+\in \mathcal{J}^\rho_{\rm reg}(\lambda_0)$ we have $\bar J_1\in \mathcal{J}_{\rm reg}^{L,\rho}(J_+,J_-)$.  Indeed, the chain map $\Phi_{\bar J_1}$ count curves in $\mathcal{M}_0(\bar J_1,\tau,\tau')$ which, by the definition of $\bar J_1$ and the fact that $\mathcal{L}$ and $\widehat{\mathcal{L}}$ are invariant by $\psi_1$, correspond to curves in $\mathcal{M}_0(J_+,\tau,\tau')$. By the regularity of $J_+$, the latter curves are necessarily cylinders over $\lambda_0$-Reeb chords. Therefore, if $\tau \in \mathcal{T}^{\rho,C}(\lambda_0)$, then $$\Phi_{\bar J_1}(\tau)=\tau'\in \mathcal{T}^{\rho,C}(c\lambda_0)\mbox{ where } \tau'(\cdot)=\tau(c\cdot).$$

Consider the exact symplectic cobordism $(\R \times Y, d\kappa_1:= d(e^s h_1'(s) \lambda_0))$ from $\lambda_0$ to $c\lambda_0$. Again $\mathcal{L}$ and $\widehat{\mathcal{L}}$ are exact Lagrangian cobordisms from $\Lambda$ to itself and from $\widehat\Lambda$ to itself, respectively.

We assume there exists an isotopy $(\mathbb{R} \times Y,d\kappa_t)$, $t\in [0,1]$, of exact symplectic cobordisms from $\lambda_0$ to $c\lambda_0$. For each $t\in [0,1]$ $\mathcal{L}$ and $\widehat{\mathcal{L}}$ are exact Lagrangian cobordisms in $(\mathbb{R} \times Y,d\kappa_t)$ from $\Lambda$ to itself and from $\widehat \Lambda$ to itself, respectively.

We introduce a space of homotopies of almost complex structures. We denote by $\widehat{\mathcal{J}}(\bar J_0,\bar J_1)$ the space of smooth homotopies $\bar J_t \in  \mathcal{J}^L(J_+,J_-), t\in [0,1],$ where each $\bar J_t$ is $d\kappa_t$-compatible.

For $\tau\in \mathcal{T}^{\rho,C}(\lambda_0)$, $\tau' \in \mathcal{T}^{\rho,C}(c\lambda_0)$ and $k\geq 0$ we consider the moduli spaces
\begin{equation}
\widehat{\mathcal{M}}^{\rho,C}_k(\bar J_t, \tau,\tau')= \{(t,\tilde w) : t\in[0,1] \mbox{ and }  \tilde w \in \mathcal{M}^{\rho,C}_{k-1}(\bar J_t, \tau,\tau')\},
\end{equation} where each $\tilde w\in \mathcal{M}^{\rho,S}_{k-1}(\bar J_t, \tau,\tau')$ maps the boundary components $\{0\}\times \R$ and $\{1\} \times \R$ into $\mathcal{L}$ and $\widehat{\mathcal{L}}$, respectively.

Using results from \cite{Ab1,AD,Rizzel,Dr,Momin}, one obtains that for a generic subset $\widehat{\mathcal{J}}_{\rm reg}^{\rho}(\bar J_0,\bar J_1)$ of $\widehat{\mathcal{J}}(\bar J_0,\bar J_1)$ the moduli spaces $\widehat{\mathcal{M}}^{\rho,C}_k(\bar J_t,\tau,\tau')$ are Fredholm regular for all $\tau,\tau'$ and $k\geq 0$, forming a $k$-dimensional manifold with boundary, whose boundary is $$\left(\{0\}\times \mathcal{M}^{\rho,C}_{k-1}(\bar J_0,\tau,\tau')\right) \cup \left(\{1\}\times \mathcal{M}^{\rho,C}_{k-1}(\bar J_1,\tau,\tau')\right).$$ In particular, $\widehat{\mathcal{M}}^{\rho,C}_0(\bar J_t, \tau,\tau')$ is a finite $0$-dimensional manifold and the regularity of $\bar J_0$ and $\bar J_1$ implies that pairs $(t,\tilde w)\in \widehat{\mathcal{M}}^{\rho,C}_0(\bar J_t, \tau,\tau')$ satisfy $t\not\in \{0,1\}$.

We define the map $$\mathfrak{K}_{\bar J_t}: {\rm LCC}^{\rho,C}(\lambda_0) \to {\rm LCC}^{\rho,C}(c\lambda_0),$$ by $$\mathfrak{K}_{\bar J_t}(\tau) := \sum_{\tau'\in \mathcal{T}^{\rho,C}(\lambda_0)} \#_{\Z_2}\widehat{\mathcal{M}}^{\rho,C}_0(\bar J_t, \tau,\tau')\tau'.$$

\begin{prop} \label{propchainhom} Under the conditions above, the map $\mathfrak{K}_{\bar J_t}$ is well defined and $\bar \Phi_{\bar J_0}$ is chain homotopic to $\Phi_{\bar J_1}$, i.e.,  \begin{equation}\label{alghomotopy}\Phi_{\bar J_0}-\Phi_{\bar J_1} =\mathfrak{K}_{\bar J_t} \circ d^{\rho}_{J_+}  + d^{\rho}_{J_-} \circ \mathfrak{K}_{\bar J_t}.\end{equation} In particular, the maps $\bar \Phi_{\bar J_0}$ and $\bar \Phi_{\bar J_1}$, defined on the homology level, coincide. %if ${\rm LCH}^{\rho,S}(J_+)\neq 0$, then $\Phi_{\bar J_0} : {\rm LCC}^{\rho,S}(\lambda_0) \to {\rm LCC}^{\rho,S}(c\lambda_0)$, defined in \eqref{defchainmap}, is non-trivial.
\end{prop}

\proof  To show that $\mathfrak{K}_{\bar J_t}$ is well defined, fix $\tau \in \mathcal{T}^{\rho,C}(\lambda_0)$ and $\tau'\in \mathcal{T}^{\rho,C}(c\lambda_0)$. We prove that $\widehat{\mathcal{M}}^{\rho,C}_0(\bar J_t, \tau,\tau')$ is finite. Indeed, take a sequence $(t_n,\tilde w_n)\in \widehat{\mathcal{M}}^{\rho,C}_0(\bar J_t, \tau,\tau')$. We may assume that $t_n \to t^* \in [0,1]$. Arguing as in the proof of Lemma \ref{lemmadifferential} we conclude that no sequence of bubbling-off points exists. Indeed, if this is not the case then, up to extraction of a subsequence, the SFT limit of $\tilde w_n$ contains a pseudoholomorphic plane or a pseudoholomorphic half-plane, which is $\bar J_{t^*},$ $J_+$ or $J_-$-holomorphic and does not intersect $\R \times L$, by positivity and stability of intersections. The plane is asymptotic to a contractible Reeb orbit and the half-plane is asymptotic to a contractible Reeb chord in $Y \setminus L$ either from $\Lambda$ to itself or from $\widehat \Lambda$ to itself. This contradicts conditions (a)-(d)  satisfied by $\lambda_0$ since it is adapted to $(Y\setminus L, \Lambda \to \widehat \Lambda)$. Hence no sequence of bubbling-off points exists.

Since the Fredholm index of $(t_n,\tilde w_n)$ is $0$, the SFT compactness theorem and stability and positivity of intersections tell us that up to a subsequence $(t_n,\tilde w_n)$ converges to a building $\tilde w$ with finitely many strips filtered by $\rho$ and $C$, precisely one of them being $\bar J_{t^*}$-holomorphic, all of them having Fredholm index at least $0$ and so that the total Fredholm index is $0$. Since the curves in $\tilde w$ which are $J_+$ or $J_-$-holomorphic have Fredholm index at least $1$, we conclude that $\tilde w$ consists of a single $\bar J_{t^*}$-holomorphic curve in $\mathcal{M}^{\rho,C}_{-1}(\bar J_{t^*},\tau, \tau')$ and hence $(t^*,\tilde w)\in \widehat{\mathcal{M}}^{\rho,C}_{0}(\bar J_{t},\tau, \tau')$.  By the regularity of $\bar J_0$ and $\bar J_1$, we know that $t^*\not\in \{0,1\}$. Since $\widehat{\mathcal{M}}^{\rho,C}_{0}(\bar J_{t},\tau, \tau')$ is a $0$-dimensional manifold $(t_n,\tilde w_n)$ must be eventually constant equal to $(t^*,\tilde w)$ which implies that $\mathcal{M}^{\rho,C}_0(\bar J,\tau, \tau')$ is finite. Now condition (e) and the fact that $\Phi_{\bar J}$ decreases action  imply that for fixed $\tau$, there exists only finitely many $\tau'$ for which $\widehat{\mathcal{M}}^{\rho,C}_0(\bar J_t,\tau, \tau')$ is non-empty. Hence $\mathfrak{K}_{\bar J_t}$ is well defined.

The proof of \eqref{alghomotopy} is a combination of the SFT compactness theorem and the gluing theorem. Let $I\subset \widehat{\mathcal{M}}^{\rho,S}_1(\bar J_t, \tau,\tau')$ be  a connected component homeomorphic to an interval and let $(t_n,\tilde w_n) \in I$ be a sequence converging to one of its ends or to one of its boundary components. Arguing as in the proof of Lemma \ref{lemmadifferential} we conclude that no bubbling-off point exists. Indeed, if this is not the case, then, up to extraction of a subsequence, we may assume that $t\to t^*\in [0,1]$ and the SFT limit of $\tilde w_n$ contains a pseudoholomorphic plane or a pseudoholomorphic half-plane which is $\bar J_{t^*},$ $J_+$ or $J_-$-holomorphic and does not intersect $\R \times L$, by positivity and stability of intersections. The plane is asymptotic to a contractible Reeb orbit and the half-plane is asymptotic to a contractible Reeb chord in $Y\setminus L$ either from $\Lambda$ to itself or from $\widehat \Lambda$ to itself. We know that this contradicts conditions (a)-(d) satisfied by $\lambda_0$. Hence, from stability and positivity of intersections, the SFT compactness theorem and condition (e) imply the existence of a subsequence of $(t_n,\tilde w_n)$ converging to $(t^*,\tilde w)$ for some $t^*\in [0,1]$, where $\tilde w$ is a pseudoholomorphic building with $q$-levels, each level consisting of a somewhere injective finite energy strip $\tilde w^l$ so that, for some $1\leq s_0\leq q$, we have
\begin{itemize}
\item if $1\leq l\leq s_0-1$, then $\tilde w^l\in \widetilde{\mathcal{M}}_{k_l}^{\rho,C}(J_+,\tau_{l},\tau_{l+1})$ for some  $k_l\geq 1,\tau_l \in\mathcal{T}^{\rho,C}(\lambda_0)$.
\item $\tilde w^{s_0}\in \mathcal{M}_{k_{s_0}}^{\rho,C}(\bar J_{t^*},\tau_{s_0},\tau_{s_0+1})$ for some  $k_{s_0}\geq -1,\tau_{s_0} \in\mathcal{T}^{\rho,C}(\lambda_0)$.
\item if $s_0+1\leq l\leq q$, then $\tilde w^l\in \widetilde{\mathcal{M}}_{k_l}^{\rho,C}(J_-,\tau_{l},\tau_{l+1})$ for some  $k_l\geq 1,\tau_l \in\mathcal{T}^{\rho,C}(c\lambda_0)$,
\end{itemize}
where $\tau_1 = \tau$ and $\tau_{k+1} = \tau'$.

Since $(t^*,\tilde w)$ is the SFT limit of a sequence of pseudoholomorphic strips of Fredholm index 1 and the homotopy $\{\bar J_t,t\in [0,1]\}$ is regular, we have $$1=I_F(\widetilde{w})=\sum_{l=1}^q I_F(\tilde w_l) = k_{s_0}+1+ \sum_{l\neq s_0}k_l  \geq -1+1 + q-1=q-1.$$ Hence $q\in \{1,2\}$.

If $q=1$, then $s_0=1$, $k_1=0$ and $(t^*,\tilde w) \in \widehat{\mathcal{M}}^{\rho,C}_1(\bar J_{t^*}, \tau,\tau')$. In this case, we necessarily have $t^*\in \{0,1\}$ since, otherwise, by the regularity of the homotopy $\bar J_t,t\in [0,1],$ if $t^*\in (0,1)$ then $(t^*,\tilde w)$ lies in the interior of $I$. This is a  contradiction with the fact that $(t_n,\tilde w_n)$ is converging to a boundary point of $I$. We conclude that \begin{equation}\label{sita} \tilde w \in \mathcal{M}^{\rho,C}_0(\bar J_{t^*},\tau,\tau') \mbox{ where } t^*\in \{0,1\}.\end{equation} Observe that such a pseudoholomorphic strip $\tilde w$ enters in the computation of $\Phi_{\bar J_{t^*}}(\tau)$.

Now assume that $q=2$. If $s_0=1$, then $k_1=-1$, $k_2=1$ and  \begin{align}\label{sito1} \tilde w^1 \in \mathcal{M}_{-1}^{\rho,C} (\bar J_{t^*},\tau,\tau_2)\mbox{ and }\tilde w^2 \in \widetilde{\mathcal{M}}^{\rho,C}_1 (J_-,\tau_2,\tau') \\ \nonumber \mbox{ for some }\tau_2 \in \mathcal{T}^{\rho,S}(c\lambda_0). \end{align} By the regularity of $\bar J_0$ and $\bar J_1$, we necessarily have $t^*\not\in\{0,1\}$ . Observe that the $2$-building formed by $\tilde w_1$ and $\tilde w_2$ in \eqref{sito1} enters in the computation of $d^{\rho}_{\bar J_-} \circ \mathfrak{K}_{\bar J_{t}}(\tau)$.

If $s_0=2$, then $k_1=1$, $k_2=-1$ and  \begin{align}\label{sito2} \tilde w^1 \in \widetilde{\mathcal{M}}^{\rho,C}_1 (J_+,\tau,\tau_2) 
 \mbox{ and }\tilde w^2 \in \mathcal{M}_{-1}^{\rho,C} (\bar J_{t^*},\tau_2,\tau') \\ \nonumber \mbox{ for some }\tau_2 \in \mathcal{T}^{\rho,C}(\lambda_0). \end{align} By the regularity of $\bar J_0$ and $\bar J_1$, we necessarily have $t^*\not\in\{0,1\}$. Observe that the $2$-building formed by $\tilde w_1$ and $\tilde w_2$ in \eqref{sito2} enters in the computation of $\mathfrak{K}_{\bar J_{t}}\circ d^{\rho}_{J_+}(\tau)$.

We conclude that to each boundary point of $ \widehat{\mathcal{M}}_1^{\rho,S} (\bar J,\tau,\tau')$ we can associate:
\begin{itemize}
    \item either an element of $\widetilde{\mathcal{M}}^{\rho,S}_0(\bar J_{t^*},\tau,\tau')$ as in \eqref{sita} which is counted either by $\Phi_{\bar J_{0}}$ or by $\Phi_{\bar J_{1}}$,
    \item or a $2$-level pseudoholomorphic building either as in \eqref{sito1} or as in \eqref{sito2}, which are counted by $d^{\rho}_{\bar J_-} \circ \mathfrak{K}_{\bar J_{t}}(\tau)$ or $\mathfrak{K}_{\bar J_{t}}\circ d^{\rho}_{\bar J_+}(\tau)$, respectively.
\end{itemize}
On the other hand the gluing theorem describes a neighbourhood of 2-level buildings as in \eqref{sito1} and \eqref{sito2}. This neighbourhood is homeomorphic to the interval $[0, +\infty)$, taking $0$ to the 2-level building and all other points to elements in  $\widehat{\mathcal{M}}^{\rho,C}_1 (\bar J_t,\tau,\tau')$.

The compactification of $\widehat{\mathcal{M}}_1^{\rho,S} (\bar J_t,\tau,\tau')$ has the structure of a $1$-dimensional compact manifold with boundary and the number of boundary components of this $1$-dimensional manifold is even. The discussion above shows that using $\Z_2$-coefficients, $(\Phi_{\bar J_0}-\Phi_{\bar J_1}-\mathfrak{K}_{\bar J_t} \circ d^{\rho}_{J_+}  -d^{\rho}_{J_-} \circ \mathfrak{K}_{\bar J_t})(\tau)$ coincides with the number of boundary points of $\widehat{\mathcal{M}}_1^{\rho,S} (\bar J_t,\tau,\tau')$, for fixed $\tau \in \mathcal{T}^{\rho,C}(\lambda_0)$ and varying $\tau'\in \mathcal{T}^{\rho,C}(c\lambda_0)$. This number is zero proving \eqref{alghomotopy}. The proposition follows. \qed

\section{Homotopical growth rate of Legendrian contact homology and forcing of topological entropy} \label{section4}

In this section we consider a quintuple $(Y,\xi,L,\Lambda,\widehat \Lambda)$, where $(Y,\xi)$ is a closed connected contact $3$-manifold, $L$ is a transverse link and $\Lambda,\widehat \Lambda\subset Y \setminus L$ are disjoint closed Legendrian knots on $(Y,\xi)$.

Let $\lambda_0$ be a contact form on $(Y,\xi)$ adapted to $(Y\setminus L, \Lambda \to \widehat \Lambda)$. In particular $L$ is a set of Reeb orbits of $\lambda_0$.
 For each $C>0$, we denote by $\Omega^C(\lambda_0, \Lambda \to \widehat \Lambda)\subset \pi_1(Y\setminus L, \Lambda \to \widehat \Lambda)$ the subset of homotopy classes $\rho$ of paths in $Y \setminus L$ from $\Lambda$ to $\widehat \Lambda$ for which
\begin{itemize}
\item[(i)] all Reeb chords of $R_{\lambda_0}$ in the class $\rho$ are transverse and have action $\leq C$.
\item[(ii)] ${\rm LCH}^{\rho,C}(J)$ is well defined and does not vanish for all $J \in \mathcal{J}_{\rm reg}(\lambda_0)$.
\end{itemize}
Notice that conditions (i) and (ii) above imply that ${\rm LCH}^{\rho,C^*}(J)\simeq{\rm LCH}^{\rho,C}(J)\neq 0$  for all $C^*>C$.

\begin{defn}
Let $\lambda_0$ be a contact form on $(Y,\xi)$ which admits $L$ as a set of Reeb orbits and moreover is adapted to $(Y\setminus L,\Lambda \to \widehat \Lambda)$. The number
\begin{equation}
\limsup_{C \to +\infty} \frac{\log (\#\Omega^C(\lambda_0, \Lambda \to \widehat \Lambda))}{C}
\end{equation}
is called the exponential homotopical growth rate of $LCH_L(\lambda_0,\Lambda \to \widehat \Lambda)$. When $\limsup_{C \to +\infty} \frac{\log (\#\Omega^C(\lambda_0, \Lambda \to \widehat \Lambda))}{C}>0$, we say that $LCH_L(\lambda_0,\Lambda \to \widehat \Lambda)$ has exponential homotopical growth. 
\end{defn}

We now proceed to prove theorem \ref{theorem:growthleg}.

\subsection{Proof of Theorem \ref{theorem:growthleg}.}

%\begin{thm} \label{theorem:growthleg1}Assume that there exists a contact form $\lambda_0$  on $(Y,\xi)$ which admits $L$ as a set of Reeb orbits and, moreover, is adapted to $(Y\setminus L,\Lambda \to \widehat \Lambda)$, and that $LCH_L(\lambda_0,\Lambda \to \widehat \Lambda)$ has positive exponential homotopical growth rate. Then, the link $L$ forces topological entropy in $(Y,\xi)$. 

%Moreover, if the exponential homotopical growth rate of $LCH_L(\lambda_0,\Lambda \to \widehat \Lambda)$ is denoted by $a$, then for every contact form $\lambda$ on $(Y,\xi)$ which has $L$ as a set of Reeb orbits, we have
%\begin{equation}    h_{top}(\phi_\lambda) \geq \frac{a}{\max f_\lambda}, \end{equation}
%where $f_\lambda$ is the function such that $\lambda=f_\lambda \lambda_0$. \end{thm}

\begin{defn}We denote by $N^C(\lambda,\Lambda \to \widehat \Lambda)$ the number of distinct Reeb chords of $X_\lambda$ from $\Lambda$ to $\widehat \Lambda$ and which have action $\leq C$. 
\end{defn}

By Weinstein's tubular neighbourhood theorem for Legendrian submanifolds, see Geiges \cite{GeigesBook},  $\widehat \Lambda$ has a tubular neighbourhood $\mathcal{V} \subset Y$, disjoint from $L$ and $\Lambda$, such that $(\mathcal{V},\xi|_{\mathcal{V}})$ is contactomorphic to the solid torus $(S^1 \times \mathbb{D},\ker (dx + \theta d y))$, where $\theta \in S^1$ and $z=(x,y)\in \mathbb{D}$. In such coordinates,  $\widehat \Lambda \equiv S^1 \times \{0\}$. We denote by $\widehat \Lambda^z,z\in \D,$ the Legendrian knot $S^1 \times \{z\}$.  Below we consider the fibration of $\mathcal{V}\to \D$ induced by $\widehat \Lambda^z$. We call such a neighbourhood $\mathcal{V}$ together with a fixed choice of contactomorphism to $(S^1 \times \mathbb{D},\ker (dx + \theta d y))$ a parametrized Legendrian tubular neighbourhood of $\widehat{\Lambda}$.
%We will start by presenting a result which shows how to estimate the growth of $N_C(\alpha,\Lambda,\Lambda')$ using the strip Legendrian contact homology.

We introduce some more terminology. We will consider the Lebesgue measure on the unit disk $\mathbb{D}$. Therefore, when we say that a set $A \subset B$ has full measure in a subset $B \subset \mathbb{D}$, we mean this with respect to this measure.
Moreover, for each $1>\epsilon>0$ we let $ \overline{\mathbb{D}}_\epsilon$ be the set $\{z \in \mathbb{D} \ | \ |z| \leq \epsilon \}$.

The following uniform estimate will be crucial to obtain positivity of topological entropy for a Reeb flow on $(Y,\xi)$ admitting $L$ as a set of Reeb orbits. 

\begin{prop} \label{mainproposition}Assume that $LCH_L(\lambda_0,\Lambda \to \widehat \Lambda)$ has positive exponential homotopical growth rate, which we denote by $a$. Let $\lambda$ be a contact form on $(Y,\xi)$ so that all knots in $L$ are $\lambda$-Reeb orbits. Let $\mathcal{V}\subset Y\setminus L$ be a parametrized Legendrian tubular neighbourhood of $\widehat{\Lambda}$.  Then, given $1>\epsilon>0$, there exist numbers $\delta>0$, a  subset $\mathcal{U}_\delta \subset \overline{\D}_\delta$ of full measure, $d$ and an increasing sequence $C_n\to +\infty$ such that $$\min_{|z|\leq\delta} \{N^{C_n}(\lambda,\Lambda \to \widehat \Lambda^z)\}>e^{\frac{a}{{\max f_\lambda - \epsilon}}C_n+d},$$
for all $\widehat{\Lambda}^z$ with $z\in \mathcal{U}_\delta$, and every $C_n$ of the sequence.
\end{prop}

Because of its length, we divide the proof of the proposition in several lemmas.

\begin{lem} \label{lemma:lowerbound}
Assume that $LCH_L(\lambda_0,\Lambda \to \widehat \Lambda)$ has positive exponential homotopical growth rate, which we denote by $a$. Let $\lambda$ be a non-degenerate contact form on $(Y,\xi)$ so that all knots in $L$ are $\lambda$-Reeb orbits, and such that all Reeb chords in $\mathcal{T}(\lambda_0,\Lambda \to \widehat{\Lambda})$ are non-degenerate. Then
\begin{equation} \label{desig1}
N^{C}(\lambda,\Lambda \to \widehat \Lambda) \geq  {\#\Omega^{\frac{C}{\max f_\lambda}}(\lambda_0,\Lambda \to \widehat \Lambda)},
\end{equation}
where $f_\lambda$ is the function such that $\lambda = f_\lambda \lambda_0$.
\end{lem}

\proof Since $N^C(\lambda,\Lambda \to \widehat \Lambda^z)=N^{kC}(k\lambda,\Lambda \to \widehat \Lambda^z),\forall k>0$, it is enough to establish the lemma in the case that the  function $f_\lambda$ satisfies $0<f_\lambda<1$.

%Since $a= \limsup_{C \to +\infty} \frac{\log (\#\Sigma^C(\lambda_0, \Lambda \to \widehat \Lambda))}{C}$, it then follows that for any $\epsilon >0$, there exists $d$ and an increasing sequence $C_n \to \infty$ such that $\#\Sigma^C(\lambda_0, \Lambda \to \widehat \Lambda)) \geq e^{(a-\epsilon)C_n+d}$. This combined with \eqref{desig1} gives which proves the proposition in the particular case of $z=0$. The second step is to show that this estimate holds and is uniform in $z$ for $|z|$  sufficiently small.

Fix $C>0$, $\rho \in \Omega^C(\lambda_0,\Lambda \to \widehat \Lambda^0)$ and $0<c<\min f<1$ small. We shall consider two different symplectic cobordisms from $\lambda_0$ to $c \lambda_0$.

Let $h_1:\R \to \R$ be a smooth diffeomorphism satisfying $h_1(s)= s/c, \forall s\leq -R-1,$ and $h_1(s) = s, \forall s\geq R+1$. Choose $J_+\in \mathcal{J}_{\rm reg}(\lambda_0)$ so that ${\rm LCH}^{\rho,C}(J_+)\neq 0$ and consider the almost complex structure $$\bar J_1:= \psi_1^* J_+,\mbox{ where } \psi_1=h_1\oplus {\rm Id}\mbox{ on } \R \times Y.$$   Note that $\bar J_1\in \mathcal{J}^L(J_+,J_-)$, where $J_-\in \mathcal{J}_{\rm reg}(c\lambda_0)$ satisfies $J_+|_\xi = J_-|_\xi$. Moreover, we can choose an exact symplectic form $d\varsigma_1$ satisfying $\varsigma_1 = e^s\lambda_0$ on $[R+1,+\infty)$ and $\varsigma_1= ce^s\lambda_0$ on $(-\infty,-R-1]$ so that $\bar J_1$ is compatible with $d\varsigma_1$.

The diffeomorphism $\psi_1$ maps $\bar J_1$-holomorphic strips to $J_+$-holomorphic strips.  The regularity of $J_+$ implies that $$\bar J_1 \in \mathcal{J}_{\rm reg}^L(J_+,J_-).$$ In particular, $\bar J_1$-holomorphic strips with Fredholm index $0$ are mapped to cylinders over $\lambda_0$-Reeb chords and thus the chain map $\Phi_{\bar J_1}$ is a reparametrization of Reeb chords, i.e, if $\tau \in \mathcal{T}^{\rho,C}(\lambda_0)$ then $\Phi_{\bar J_1}(\tau) \in \mathcal{T}^{\rho,cC}(c\lambda_0)$, where $\Phi_{\bar J_0}(\tau) = \tau (c \cdot)$.

Since there are no $\lambda_0$-Reeb chords in class $\rho$ with action $>C$, we see that \begin{equation}\label{ntrivial1}{\rm LCH}^{\rho,C}(J_+)\simeq {\rm LCH}^{\rho,C/c}(J_+)\simeq {\rm LCH}^{\rho,C}(J_-)\neq 0.\end{equation} Hence the induced map on the homology level \begin{equation}\label{ntrivial2}\bar \Phi_{\bar J_1}:{\rm LCH}^{\rho,C}(J_+)\to {\rm LCH}^{\rho,C}(J_-)\mbox{ is an isomorphism}.\end{equation}

In order to construct another exact symplectic cobordism between $\lambda_0$ and $c\lambda_0$, fix a smooth function $g_0:\R \times Y \to [c,1]$ satisfying
\begin{itemize}
\item $g_0(s,x)=c,\forall (s,x)\in (-\infty,-2]\times Y$.
\item $g_0(s,x)=f(x), \forall (s,x)\in [-1,1]\times Y$.
\item $g_0(s,x)=1,\forall (s,x) \in [2,+\infty)\times Y$.
\item $\partial_s g_0(s,x) \geq 0,\forall (s,x)\in \R \times Y$.
\end{itemize}
For each $R>0$, let $h_R:\R \times Y \to [c,1]$ be given by
\begin{itemize}
\item $h_R(s,x)=g_0(s+R-1,x), \forall (s,x)\in (-\infty,-R] \times Y$.
\item $h_R(s,x)=f(x), \forall (s,x) \in [-R,R] \times Y$.
\item $h_R(s,x)=g_0(s-R+1,x), \forall (s,x)\in [R,+\infty)$.
\end{itemize}
Observe that $h_R$ is smooth by definition and that $d\varsigma_R$ is an exact symplectic form on $\R \times Y$, where $\varsigma_R:= e^sh_R(s,\cdot)\lambda_0$.

Let $J_\lambda \in \mathcal{J}_{\rm reg}(\lambda)$. We choose a $d\varsigma_R$-compatible $\bar J_R\in \mathcal{L}^L(J_+,J_-)$ so that $\bar J_R|_{[-R,R]\times Y}=J_\lambda$. It induces a chain map $\Phi_{\bar J_R}$ which is non-trivial. In fact, if $\Phi_{\bar J_R}\equiv 0$, then  $\Phi_{\bar J_R}$ is Fredholm regular. Later we shall use the dependence of $\bar J_R$ on $R$.

Consider an isotopy $(\mathbb{R} \times Y,d\varsigma_t)$, $t\in [0,1]$, of exact symplectic cobordisms from $\lambda_0$ to $c\lambda_0$ which starts at $d\varsigma_R$ and ends at $d\varsigma_1$. Assume that $\varsigma_t$ has the form $h_t(s)e^s\lambda_0$ for some smooth increasing function $h_t:\R \to \R$ satisfying  $h_t = 1$ on $[R+1,+\infty)$ and $\varsigma_t = c$ on $(-\infty,-R-1]\times Y$ for all $t\in [0,1]$.

Take $\{\bar J_t,t\in [0,1]\} \in \widehat{\mathcal{J}}(\bar J_0:=J_R,\bar J_1)$ so that $\bar J_t=J_+$ on $[R+1,+\infty) \times Y$ and $\bar J_t = J_-$ on $(-\infty,-R-1]\times Y$ for all $t\in[0,1]$.

We claim that there exist $\tau\in \mathcal{T}^{\rho,C}(\lambda_0), \tau'\in \mathcal{T}^{\rho,C}(c\lambda_0)$ and $k\geq 0$ so that $$\mathcal{M}^{\rho,C}_k(\bar J_R,\tau,\tau')\neq \emptyset.$$ Indeed, if no such $\tau,\tau'$ and $k$  exist, then  $\bar J_R \in \mathcal{J}^{L,\rho}_{\rm reg}(J_+,J_-)$ and the chain map $\Phi_{\bar J_R}: {\rm LCC}^{\rho,C}(\lambda_0)\to {\rm LCC}^{\rho,C}(c\lambda_0)$, defined as in \eqref{defchainmap}, vanishes. The induced map $\bar \Phi_{\bar J_R}$ on the homology level vanishes as well. However, by Proposition \ref{propchainhom}, $\Phi_{\bar J_R}$ is chain homotopic to $\Phi_{\bar J_1}$ and we conclude that $0=\bar \Phi_{\bar J_R}=\bar \Phi_{\bar J_1}$, contradicting \eqref{ntrivial1} and \eqref{ntrivial2}.

We now take an increasing sequence $R_n \to +\infty$, and a sequence of pseudoholomorphic strips $\widetilde w_n \in \mathcal{M}^{\rho,C}(\bar J_{R_n},\tau,\tau')$. Because there is a bound on the energy of all $\widetilde w_n$, the SFT-compactness theorem there exists a subsequence of $\widetilde w_n$ that converges to a pseudoholomorphic building that we will denote by $\widetilde{w}$. We proceed to analyse the structure of the pseudoholomorphic building $\widetilde{w}$. Since the topology of the domain of a pseudoholomorphic curve does not change when there is a breaking, the domain of the pseudoholomorphic building is a broken strip in the sense of \cite{CPT}.

 Let $\widetilde{w}^l$ for $l\in \{1,...,q\}$ be the levels of the pseudoholomorphic building $\widetilde{w}$. Because of positivity and stability of intersections the levels of $\widetilde{w}$ cannot intersect $\R \times L$. Reasoning as in \cite{Momin} we conclude that no level of $\widetilde{w}$ can have an interior puncture asymptotic to a Reeb orbit contained in $L$. Because the domain of $\widetilde{w}$ is a broken strip, the existence of such an interior puncture would imply that there is a pseudoholomorphic building $\widetilde{u}$ formed by pseudoholomorphic curves contained in $\widetilde{w}$, whose domain would be a broken disk, and whose upper level would have only one positive puncture asymptotic to a Reeb orbit in $L$. Since every disk in $Y$ whose boundary is in $L$ has an interior intersection point with $L$, we conclude that $\widetilde{u}$ would have an interior intersection point with $L$, and that would also be true for $\widetilde{w}$. But positive and stability of interior intersection points of pseudoholomorphic curves imply that  $\widetilde{w}$ cannot have an interior intersection point with $L$.
 
This together with the face that the domain of $\widetilde{w}$ is a broken strip, implies that we  must have the following picture.
\begin{itemize}
\item{The upper level $\widetilde{w}^1$ is composed of one pseudoholomorphic disc, with one positive puncture, which is asymptotic to the Reeb chord $\tau$, and several negative boundary and interior punctures. At all its negative punctures this curve is asymptotic to contractible Reeb orbits in $Y\setminus L$,  Reeb chords from $\Lambda$ to itself that represent the trivial element of  $\pi_1(Y,\Lambda)$, or Reeb chords from $\widehat{\Lambda}$ that represent the trivial element of $\pi_1(Y,\widehat{\Lambda})$, with the exception of one negative boundary puncture at which the curve is asymptotic to a Reeb chord $\tau_1$ of $\lambda_0$ or $\lambda$.}
\item{For $l\in \{2,...,q\}$ the level $\widetilde{w}^l$ contains a special curve that has one positive puncture at which it is asymptotic to a Reeb chord $\tau_{l-1}$, of $\lambda_0$, $\lambda$ or $c\lambda_0$, in the homotopy class $\rho$, and possibly several interior and boundary negative punctures. Of the negative boundary punctures there is one at which the curve is asymptotic to a Reeb chord $\tau_l$, of $\lambda_0$, $\lambda$ or $c\lambda_0$, in the homotopy class $\rho$, and at all other negative punctures the curve is asymptotic to contractible Reeb orbits in $Y\setminus L$,  Reeb chords from $\Lambda$ to itself that represent the trivial element of  $\pi_1(Y\setminus L,\Lambda)$, or Reeb chords from $\widehat{\Lambda}$ that represent the trivial element of $\pi_1(Y \setminus L,\widehat{\Lambda})$.}
\end{itemize}
As a consequence we obtain that the level $\widetilde{w}^l$ living in the exact symplectic cobordism from $\lambda$ to $c\lambda_0$ contains a pseudoholomorphic curve with one positive puncture asymptotic to a Reeb chord $\widehat{\tau} \in \mathcal{T}^{\rho}_{\Lambda \to \widehat{\Lambda}}(\lambda)$. The action of $\widetilde{\tau}$ must be smaller than the action of $\tau$, which is $C$. 

We conclude that every $\rho \in \Omega^C(\lambda_0,\Lambda \to \widehat \Lambda)$, has  a Reeb chord of $\lambda$ with action $\leq C$. Since all these Reeb chords are belong to different homotopy classe we obtain that if $0<f_\lambda <1$ 
\begin{equation} \label{desig13}
N^{C_n}(\lambda,\Lambda \to \widehat \Lambda) \geq  \#\Omega^C(\lambda_0,\Lambda \to \widehat \Lambda).
\end{equation}
As observed previously, the case of general $\lambda$ follows easily from this one.
\qed

We now recall the following results from \cite{A2}.
\begin{lem} \cite[Lemma 6]{A2} \label{generic}
Let $\lambda$ be a contact form on $(Y,\xi)$ such that $L$ is a set of Reeb orbits of $\lambda$, and $\Lambda$ and $\widehat{\Lambda}$ be disjoint connected Legendrian submanifolds in $(Y \setminus L,\xi)$. We consider a parametrized Legendrian tubular neighbourhood $\mathcal{V} \subset Y \setminus L$ of $\widehat{\Lambda}$.
Then, there exists a set $\mathcal{U} \subset \D$ of full measure, such that for every $z \in \mathcal{U}$ the Reeb chords in  $\mathcal{T}(\lambda, \Lambda \to \widehat{\Lambda}^z)$ are all transverse.
\end{lem}

Before stating the next result we introduce some terminology. We consider the space $\mathrm{Diff}^1(Y)$ of $C^1$-diffeomorphisms of $Y$ a distance function that generates canonical topology on  $\mathrm{Diff}^1(Y)$. When we say that two elements in  $\mathrm{Diff}^1(Y)$ are $\epsilon_0$ close, we mean it with respect to this distance function.
\begin{lem} \label{lem:change_of_position}\cite{A2}
Let $\lambda$ be a contact form on $(Y,\xi)$ such that $L$ is a set of Reeb orbits of $\lambda$, and $\Lambda$ and $\widehat{\Lambda}$ be disjoint connected Legendrian submanifolds in $(Y \setminus L,\xi)$. Let $\mathcal{V} \subset Y \setminus L$ be a parametrized Legendrian tubular neighbourhood of $\widehat{\Lambda}$ that does not intersect $\Lambda$. Then, given $\epsilon_0>0$ there exists $\delta>0$ such that for every $z \in \overline{\D}_\delta$ there exists a contactomorphism $\psi_{\widehat{\Lambda}^z}: (Y,\xi) \to (Y,\xi)$ which satisfies
\begin{itemize}
\item [(1)] $\psi_{\widehat{\Lambda}^z}(\widehat{\Lambda}) = \widehat{\Lambda}^z$,
\item [(2)] $\psi_{\widehat{\Lambda}^z}$ is $\epsilon_0$-close to the identity in the $C^1$-sense,
\item [(3)] $\psi_{\widehat{\Lambda}^z}$ coincides with the identity in the complement of $\mathcal{V}$.
\end{itemize}
\end{lem}

This lemma is follows trivially from \cite[Lemma 5]{A2}.
We now proceed to prove Proposition \ref{mainproposition}.

\textit{Proof of proposition \ref{mainproposition}:} \\
Let $\lambda$ be a contact form satisfying the assumptions of the proposition. Fix $\epsilon>0$. We choose $\epsilon_0>0$ such that, if  $\psi: (Y,\xi) \to (Y,\xi)$ is a contactomorphism which is $\epsilon_0$-close to the identity in the $C^1$-sense, then we have $$|f_{\psi^* \lambda}-f_\lambda|_{C^0} < \epsilon,$$
where $f_{\psi^* \lambda}$ is the function such that $f_{\psi^* \lambda}\lambda_0 = \psi^* \lambda$. It is clear that such an $\epsilon_0$ exists. 

We then apply lemma \ref{lem:change_of_position} to obtain a $\delta>0$ so that for every $z \in \overline{\D}_\delta$ there exists $\psi_{\widehat{\Lambda}^z}: (Y,\xi) \to (Y,\xi)$ satisfying (1), (2), and (3) as in the lemma with $\epsilon_0$ as in the paragraph above. It follows that if we define $\lambda_z:=\psi_{\widehat{\Lambda}^z}^* \lambda$ then \begin{equation} \label{bottleneck}
    |f_{\lambda_z} - f_\lambda|_{C^0} < \epsilon,
\end{equation} 
where $f_{\lambda_z}$ is the function that satisfies $\lambda_z=f_{\lambda_z}\lambda_0$.

For this choice of $\delta>0$ we consider the set $\mathcal{U}_\delta:= \mathcal{U}\cap \overline{\D}_\delta$, where $\mathcal{U}$ is given by lemma \ref{generic}.

For each $z \in \mathcal{U}_\delta$ we consider the contact form $\lambda_z:=\psi_{\widehat{\Lambda}^z}^* \lambda$. By the definition of $\lambda_z$, and because $\psi_{\widehat{\Lambda}^z}(\Lambda)= \Lambda$ and $\psi_{\widehat{\Lambda}^z}(\widehat{\Lambda}) = \widehat{\Lambda}^z$, it is clear that $\psi_{\widehat{\Lambda}^z}^{-1}$ takes $\lambda$-Reeb chords from $\Lambda$ to $\widehat{\Lambda}$ to $\lambda_z$-Reeb chords from $\Lambda$ to $\widehat{\Lambda}^z$. There is thus a natural action preserving bijection between $\mathcal{T}(\lambda, \Lambda \to \widehat{\Lambda}^z)$ and $\mathcal{T}(\lambda_z, \Lambda \to\widehat{\Lambda})$. In particular, we know that all Reeb chords of $\mathcal{T}(\lambda_z, \Lambda \to\widehat{\Lambda})$ are transverse and 
\begin{equation}
    N^C(\lambda_z,\Lambda \to \widehat \Lambda)= N^C(\lambda,\Lambda \to \widehat \Lambda^z).
\end{equation}
We remark that because $\psi_{\widehat{\Lambda}^z}$ is the identity in the complement of $\mathcal{V}$, it follows that $L$ is also a set of Reeb orbits of $\lambda_z$.

Because all elements of $\mathcal{T}(\lambda_z, \Lambda \to\widehat{\Lambda})$ are transverse, we can perturb $\lambda_z$ to a non-degenerate contact form $\lambda'_z$ such that
\begin{itemize}
    \item $N^C(\lambda_z,\Lambda \to \widehat \Lambda)= N^C(\lambda'_z,\Lambda \to \widehat \Lambda)$,
    \item $L$ is a set of Reeb orbits of $\lambda'_z$,
    \item $|f_{\lambda'z} - f_\lambda|_{C^0}<\epsilon$ where $f_{\lambda'z} $ is the function such that $\lambda'_z=f_{\lambda'z}\lambda_0$.
\end{itemize}
We can then apply lemma \ref{lemma:lowerbound} to $\lambda'_z$, to obtain 
\begin{align} \label{eq:import}
    N^C(\lambda_z,\Lambda \to \widehat \Lambda)= N^C(\lambda'_z,\Lambda \to \widehat \Lambda) \geq {\#\Omega^{\frac{C}{\max f_{\lambda'_z}}}(\lambda_0,\Lambda \to \widehat \Lambda)} \geq \\ \nonumber \geq {\#\Omega^{\frac{C}{\max f_{\lambda} - \epsilon}}(\lambda_0,\Lambda \to \widehat \Lambda)}.
\end{align}
The proposition now follows from combining \eqref{eq:import} with the fact fact that $\limsup_{C \to \infty } \frac{\log \#\Omega^{\frac{C}{\max f_{\lambda} - \epsilon}}(\lambda_0,\Lambda \to \widehat \Lambda)}{C}= \frac{a}{\max f_{\lambda} - \epsilon}$. \qed

\

We can now finish the proof of theorem \ref{theorem:growthleg}.

\textit{Proof of Theorem \ref{theorem:growthleg}:} Using Proposition \ref{mainproposition} the proof is identical to the one of \cite[Theorem 1]{A2}. More precisely, one shows that $$\limsup_{t \to +\infty} \frac{\log ({\rm length}(\phi_\lambda^t(\Lambda)))}{t}\geq \frac{a}{\max f_\lambda}.$$
Applying Yomdin's theorem we conclude that $h_{\rm top}(\phi_\lambda)\geq \frac{a}{\max f_\lambda}$.
\qed

\section{Homotopical growth $CH_{{L}}$ and positivity of $h_{\rm top}$} \label{section:growth-entropy}

\subsection{Recollections of the cylindrical contact homology on the complement of a transverse link}

In this section we recall some results about the cylindrical contact homology in the complement of a transverse link. Our references for this section are \cite{HMS,Momin}. 
We begin with a definition.
\begin{defn}
Let $(Y,\xi)$ be a contact $3$-manifold, $L$ be a transverse link in $(Y,\xi)$ and $\lambda_0$ be a contact form on $(Y,\xi)$. Assume that $L$ is a collection of Reeb orbits of $\lambda_0$.
We say that $\lambda_0$ is \textit{hypertight in the complement of $L$} if
\begin{itemize}
    \item any disk in $Y$ whose boundary is a component of $L$ must have an interior intersection point with $L$,
    \item every Reeb orbit of $\lambda_0$ is non-contractible in $Y\setminus L$.
\end{itemize}
\end{defn}

The following definition is taken from \cite[Definition 1.3]{Momin}.
\begin{defn}
Let $(Y,\xi)$ be a contact $3$-manifold, $L$ be a transverse link in $(Y,\xi)$, $\lambda_0$ be a contact form on $(Y,\xi)$ and $\rho$ be a non-trivial free homotopy class of loops in $Y \setminus L$. Assume that $L$ is a collection of Reeb orbits of $\lambda_0$ and that $\lambda_0$ is hypertight in the complement of $L$. We say that $(\lambda_0,L,\rho)$ satisfy the ``proper link class'' condition (PLC) if 
\begin{itemize}
    \item for any connected component $x$ of $L$, no Reeb orbit $\gamma$ in $\rho$ can be homotoped to $x$  in $Y\setminus L$, i.e. there  is  no  homotopy $I: [0,1]\times S^1 \to Y$  with $I(0, \cdot)=\gamma$ and $I(1,\cdot)=x$ such that $I([0,1) \times S^1)\subset Y\setminus L$.
    \item every Reeb orbit of $\lambda_0$ belonging to the class $\rho$ is non-degenerate and simply covered.
\end{itemize}

\end{defn}

Momin showed that if $(\lambda_0,L,\rho)$ satisfies the PLC condition then one can define the cylindrical contact homology $\CH_{\rho | L}(\lambda_0)$ of $\lambda_0$ on the complement of $L$ for Reeb orbits on $\rho$. 

We briefly recall the construction of $\CH_{\rho | L}(\lambda_0)$. The cylindrical contact  complex $\CC_{\rho | L}(\lambda_0)$ is the $\Z_2$-vector space generated by the Reeb orbits of $\lambda_0$ that belong $\rho$. There is a $\Z_2$-grading on $\CC_{\rho | L}(\lambda_0)$. If $\gamma \in \rho$ is a Reeb orbit of $\lambda_0$ the degree  $|\gamma|$ of $\gamma$ is given by the parity of the Conley-Zehnder index with respect to any trivialization of $\xi$ over $\gamma$: the degree is then extended algebraically to all $\CC_{\rho | L}(\lambda_0)$. The differential $d$ in $\CC_{\rho | L}(\lambda_0)$ counts certain finite energy pseudoholomorphic cylinders in the symplectization of $\lambda_0$, and we obtain that $(\CC_{\rho | L}(\lambda_0),d)$ is a $\Z_2$-graded chain complex. The resulting homology is the $\Z_2$-graded vector space $\CH_{\rho | L}(\lambda_0)$. For the details of the construction we refer the reader to \cite[Section 4]{Momin} and \cite[Section 3]{HMS}.

We will use the following result which is a consequence of \cite[Theorem 1.4]{Momin}.

\begin{thm} \label{theorem-momin}
Let $(Y,\xi)$ be a contact $3$-manifold, $L$ be a transverse link in $(Y,\xi)$, $\lambda_0$ be a contact form on $(Y,\xi)$ and $\rho$ be a non-trivial free homotopy class of loops in $Y \setminus L$. Assume that $L$ is a collection of Reeb orbits of $\lambda_0$, and that $(\lambda_0, L,\rho)$ satisfy the PLC condition. Assume that all Reeb orbits of $\lambda_0$ in $\rho$ have action smaller than $T>0$ and that $\CH_{\rho | L}(\lambda_0)\neq 0$. Then every contact form $\lambda$ on $(Y,\xi)$ which has $L$ as a collection of Reeb orbits, posseses a Reeb orbit $\gamma_\rho$ in the homotopy class $\rho$ that satisfies
\begin{equation}
    \mathcal{A}(\gamma_\rho) \leq \max(f_{\lambda})T,
\end{equation}
where $f_{\lambda}$ is the positive function such that $f_\lambda\lambda_0 = \lambda$. 
\end{thm}

We need the following Lemma which is a simple consequence of Momin's definition of the cylindrical contact homology in the complement of a transverse link.

\begin{lem}
Let $(Y,\xi)$ be a contact $3$-manifold, $L$ be a transverse link in $(Y,\xi)$, $\lambda_0$ be a contact form on $(Y,\xi)$ and $\rho$ be a non-trivial free homotopy class of loops in $Y \setminus L$. Assume that $L$ is a collection of Reeb orbits of $\lambda_0$, and that $(\lambda_0, L,\rho)$ satisfy the PLC condition. Assume that there exists a finite and positive number of Reeb orbits of $\lambda_0$ that belong to the class $\rho$ and that Conley-Zehnder indexes of these orbits have the same parity. Then $\CH_{\rho | L}(\lambda_0)\neq 0$.
\end{lem}
\textbf{Proof:} Because all the generators of  $\CC_{\rho | L}(\lambda_0)$ have the same parity the differential of $d$ of $\CC_{\rho | L}(\lambda_0)$ vanishes. This implies that $\CH_{\rho | L}(\lambda_0) \cong \CC_{\rho | L}(\lambda_0)$, and under the assumptions of the lemma $\CC_{\rho | L}(\lambda_0) \neq 0$. \qed

\subsection{Homotopical growth $CH_{{L}}$ and positivity of $h_{\rm top}$} \label{section:homgrowthCH}

Let $(Y,\xi)$ be a contact $3$-manifold, $L$ be a transverse link in $(Y,\xi)$, $\lambda_0$ be a contact form on $(Y,\xi)$ such that $L$ is a collection of periodic orbits of $\lambda_0$.
We assume that $\lambda_0$ is hypertight in the complement of $L$.

Let $\Omega(Y\setminus {L})$ be the set of free homotopy classes of loops in $Y\setminus L$.
For every positive real number $T$ we define the set $\Omega^{T}_{L}(\lambda_{0})\subset \Omega(Y\setminus {L})$ such that $\rho \in \Omega^{T}_{L}(\lambda_{0})$  if $(\lambda_0,L,\rho)$ satisfies the PLC condition, every Reeb orbit of $\lambda_0$ in $\rho$ has action smaller than $T$ and $CH_{{L}}^{\rho}(\lambda_{0})\neq 0.$

\begin{defn}
The exponential homotopical growth rate of $CH_L(\lambda_0)$ is defined as the number $\limsup_{T \to +\infty} \frac{\log \# \Omega^{T}_{L}(\lambda_{0}) }{T}$. When this number is $>0$ then we say that $CH_L(\lambda_0)$ has exponential homotopical growth.
\end{defn}

We recall the following result from \cite{Alves-Cylindrical}which is necessary for us.
\begin{thm}\cite[Theorem 1]{Alves-Cylindrical} \label{t:homotopicalentropy} 
Let $\phi$ be a flow on a compact manifold $Y$ possibly with boundary. Then $$h_{top}(\phi) \geq \limsup_{T \to +\infty}\frac{\log \# N^T(\phi)}{T},$$
where $N^T(\phi)$ is the set of free homotopy classes in $Y$ that contain periodic orbits of $\phi$ with length $\leq T$.
\end{thm}
For completeness we sketch a proof of this result which is simpler than the one from \cite{Alves-Cylindrical} and which was suggested to us by Alberto Abbondandolo.

\textit{ Sketch of the proof of theorem \ref{t:homotopicalentropy}.}
We assume that $ \limsup_{T \to +\infty}\frac{\log \# N^T(\phi)}{T}>0$, since the theorem is trivial in the remaining case.

We first endow $Y$ with a Riemannian metric $g$ whose distance function we denote by $d_g$. Recall that $$h_{top}(\phi)= \lim_{\epsilon \to 0} \limsup_{T \to +\infty} \frac{\log S^\epsilon_T}{T}$$ where $S^\epsilon_T$ is the maximum number of $\epsilon,T$-separated orbits. 

The following lemma is elementary and we leave its proof as an exercise to the reader.
\begin{lem*}
 Let $\epsilon_0 := \frac{\epsilon_{\rm inj}}{4}$ where $\epsilon_{\rm inj}$ denotes the injective radius of $g$. Then, there exists $\delta>0$ such that if $\gamma_1$ and $\gamma_2$ are periodic orbits with periods $T_1,T_2$ satisfying $|T_1 - T_2|\leq \delta$  and $$\sup_{0\leq t \leq \max\{T_1,T_2\}} d_g(\gamma_1(t),\gamma_2(t)) < \epsilon_0, $$ 
 then $\gamma_1$ and $\gamma_2$ are in the same free homotopy class. 
\end{lem*}

Take $\delta>0$ as in the lemma. We denote $\overline{a}:=  \limsup_{T \to +\infty}\frac{\log \# N^T(\phi)}{T}$, and recall that we are assuming $\overline{a}>0$. From the definition of $\overline{a}$ and the fact that it is positive, it is easy to see that given 
 $\varepsilon >0$, there exists a sequence $T_n \to +\infty$ such that $$  P^{T_n + \frac{\delta}{2}}_{T_n - \frac{\delta}{2}}(\phi) \geq e^{(\overline{a}-\varepsilon)T_n}, $$ 
where $ P^{T_n + \frac{\delta}{2}}_{T_n - \frac{\delta}{2}}(\phi)$ is the number of free homotopy classes in $Y$ that contain a periodic orbits of $\phi$ with period in $[T_n - \frac{\delta}{2},T_n + \frac{\delta}{2}]$. This observation combined with the previous lemma finishes the proof of theorem \ref{t:homotopicalentropy}. \qed

We now prove theorem \ref{theorem:growth-entropy} which gives a condition for a transverse link in a contact $3$-manifold to force topological entropy.
%\begin{thm}\label{theorem:growth-entropy1}Let $(M,\xi)$ be a contact $3$-manifold, $L$ be a transverse link in $(M,\xi)$, $\lambda_0$ be a contact form on $(M,\xi)$ such that $L$ is a collection of periodic orbits of $\lambda_0$.We assume that $\lambda_0$ is hypertight in the complement of $L$ and that the cylindrical contact homology of $\lambda_0$ in the complement of $L$  has exponential homotopicalgrowth rate $a>0$. Then, if $\lambda$ is a contact form on $(M,\xi)$ which has ${L}$ as a set of Reeb orbits the topological entropy of the Reeb flow $\phi_\lambda$  is positive. Moreover$$ h_{\rm top}(\phi_{\lambda})\geq \frac{a}{\max  f_{\lambda}}, $$where $f_{\lambda}$ is the function such that $\lambda=f_{\lambda}\lambda_{0}$.\end{thm}

\subsection{ Proof of Theorem \ref{theorem:growth-entropy}.}
Let $\lambda$ be a contact form on $(Y,\xi)$ which has ${L}$ as a set of Reeb orbits.
Applying Theorem \ref{theorem-momin} we obtain that if $\rho\in \Omega_{T}(\lambda_{0}),$ then there exists a Reeb orbit $\gamma_\rho$ of $\lambda$ in $\rho$ with period less than $\frac{T}{\max f_{\lambda}}$.

Defining
$$N_{T}(\phi_{\lambda})=\#\{\rho\in \Omega(Y\setminus {L}) \ \ | \ \ \phi_{\lambda} \ \  \mbox{has a Reeb orbit with action} \leq T \ \ \mbox{in}\ \  \rho\}$$
we conclude from our reasoning in the first paragraph that 
$$N_{T}(\phi_{\lambda}) \geq e^{\frac{a}{{\rm max}f_{\lambda}}T+b}.$$

We now apply the blow-up construction explained in Section \ref{section:blowup}, and blow up $Y$ along the link $L$ to obtain a manifold $Y_L$ and a flow $\widehat{\phi}_{\lambda}$ on $Y_L$.

Define  
$$\widehat{N}_{T}(\phi_{\lambda})=\#\{\widehat{\rho}\in \Omega(Y_L\setminus \partial Y_L) \ \ | \ \ \widehat{\phi}_{\lambda} \ \  \mbox{has a Reeb orbit with action} \leq T \ \ \mbox{in}\ \  \widehat{\rho}\}.$$
We observe that the restriction $\Pi:Y_L\setminus \partial Y_L \to Y\setminus L$ of the projection $\Pi: Y_L \to Y$ is a diffeomorphism and a conjugacy between the restriction of $\phi_\lambda$ to $Y\setminus L$ and the restriction of $\widehat{\phi}_\lambda$ to $Y_P\setminus \partial Y_P$. From this it follows that
$$\widehat{N}_{T}(\phi_{\lambda})\geq N_{T}(X_{\lambda}) \geq e^{\frac{a}{{\rm max}f_{\lambda}}T+b}.$$
By theorem \ref{t:homotopicalentropy} we obtain that
$$h_{\rm top}(\widehat{\phi}_{\lambda}) \geq \limsup_{T\to +\infty} \frac{\log \widehat{N}_{T}(X_{\lambda})}{T} \geq \frac{a}{\max f_\lambda}. $$

By Theorem \ref{theorem-Bowen} we obtain that
$$h_{\rm top}(\phi_{\lambda}) = h_{\rm top}(\widehat{\phi}_{\lambda}) \geq \frac{a}{\max f_\lambda},$$ 
which concludes the proof of the theorem. \qed

\section{Examples of transverse links that force $h_{\rm top}$}\label{section:existence}

In this section we show that every contact $3$-manifold admits a transverse knot that forces entropy.

\subsection{Recollections on open books}
\label{ss:recollections}
In this paragraph we collect results on open books needed in our proof.
For more information and details we refer to \cite[Section 4.4]{GeigesBook}.

Let $Y$ be a closed connected orientable 3-manifold.
An {\it open book}\footnote{In the literature this is usually called an abstract open book decompostion.} for $Y$ is a triple pair $(\Sigma,\psi,\Psi)$,
where $\Sigma$ is a compact oriented surface with non-empty boundary $\partial \Sigma$
and where $\psi$ is a diffeomorphism of~$\Sigma$ that is the identity near the boundary, 
such that there is a diffeomorphism $\Psi$ of $Y$ and 
$$ 
Y(\psi):=\Sigma(\psi) \cup_{\id} \left(\partial \Sigma \times \overline \D \right) .
$$
Here, $\Sigma(\psi)$ denotes the mapping torus
$$
\Sigma(\psi) \,=\, \left( [0,2\pi] \times \Sigma \right) / \sim 
$$
where $(0,\psi(x)) \sim (2\pi,x)$ for each $x \in \Sigma$,
and $\overline \D$ is the closed unit disc. Viewing $S^1$ as the interval $[0,2\pi]$ with its endpoints identified, we write $\partial (\Sigma(\psi))$ as $\partial \Sigma \times S^1$.
The manifold~$Y$ is thus presented as the union of the mapping torus~$\Sigma(\psi)$
and finitely many solid tori, one for each boundary component of~$\Sigma$,
glued along their boundaries by the identity map
$$
\partial (\Sigma(\psi)) \,=\, \partial \Sigma \times S^1 \stackrel{\id\,}{\longrightarrow} 
\partial (\partial \Sigma \times \overline \D) .
$$
We remark that the diffeomorphism $$\Psi: Y(\psi) \to Y$$ is part of the definition of the open book decomposition. 
If $\psi':\Sigma \to \Sigma$ is a diffeomorphism that is isotopic to $\psi$ via an isotopy that fixes each point of $\partial \Sigma$ then $Y(\psi')$ is diffeomorphic to $Y(\psi)$. The map $\psi$ is called the monodromy map of the open book decomposition.

The page $\dot{\Sigma}_t$ of the open book decomposition is the image by the diffeomorphism $\Psi$ of the union of $\{t\} \times \Sigma$ with the half-open annuli 
$$
A_t \,=\, \partial \Sigma  \times \{ (r,t) \in \overline \D \setminus \{0\} \} .
$$
The closure of the pages $\Sigma_t$ is diffeomorphic to $\Sigma$, and their common boundary $\mathcal B$, called the binding of the open book, is the image by $\Psi$ of $\partial \Sigma \times \{0\} \subset \partial \Sigma \times \mathbb{D}$.  The orientation of $\Sigma$ induces orientations on the pages and the binding. 

Two open book decompositions $(\Sigma,\psi,\Psi)$ and $(\Sigma,\psi',\Psi')$ are diffeomorphic if $Y(\psi)$ and $Y(\psi')$ are both diffeomorphic to the same $3$-manifold $Y$, and there exists a diffeomorphism of $Y$ taking pages of  $(\Sigma,\psi,\Psi)$ to pages of  $(\Sigma,\psi',\Psi')$. 

\begin{rem}
If $\Sigma$ is a closed orientable surface with non-empty boundary and $\psi$ and $\psi'$ are diffeomorphisms of $\Sigma$ which are the identity near $\partial \Sigma$, then $Y(\psi)$ and $Y(\psi')$ are diffeomorphic if, and only if, $\psi$ and $\psi'$ are isotopic via an isotopy that is the identity on a neighbourhood of $\partial \Sigma$.
If such an isotopy exists, then for any choice of diffeomorphisms $\Psi$ and $\Psi'$ the open book decompositions $(\Sigma,\psi,\Psi)$ and $(\Sigma,\psi',\Psi')$ are diffeomorphic.
\end{rem}

\medskip \noindent
{\bf Contact structures and open books.}
Let $Y$ be a closed connected oriented 3-manifold and $Y(\psi)$ be an open book decomposition on $Y$. We will assume for simplicity that $\partial \Sigma$ is connected. %Before our next definition we introduce some special coordinates on certain regions of $M(\psi)$. 

%We first choose a collar neighbourhood $N \subset \Sigma$ of $\partial \Sigma$ on which $\psi$ is the identity. Thus, $N$ is diffeomorphic to  $[1,1+\delta] \times \partial \Sigma$, and we use this to introduce coordinates $(r,\theta)$ for $N$, where the boundary of $\partial \Sigma$ corresponds to $r=1$.  We use $N$ to obtain a collar neighbourhood $S^1 \times N$ of $\partial \Sigma(\psi)$, and use the above coordinates to obtain coordinates $(t,r,\theta)$ for $S^1 \times N$.

%Secondly, we consider on $\partial \Sigma \times \overline \D$ coordinates $(\theta,\ovr,\ot)$, where $\theta$ is our choice of coordinate for $\partial \Sigma$ and $(\ovr,\ot)$ are polar coordinates on $\overline \D$.

\begin{defn}
A contact form $\alpha$ on $Y$ is said to be {\it adapted to the open book decomposition} $(\Sigma,\psi,\Psi)$ if 
\begin{itemize}
\item[$\bullet$]
$\alpha$ is positive on~$\mathcal B$,
\item[$\bullet$]
$d\alpha$ is a positive area form on every page.
\end{itemize}
\end{defn}
It is not hard to see that a contact form $\alpha$ is adapted to an open book decomposition if, and only if, 
\begin{itemize}
    \item the Reeb vector field $X_\alpha$ is positively transverse to the interior of the pages,
    \item the Reeb vector field is tangent to the binding, and induces the positive orientation on the binding.
\end{itemize}

\begin{defn}
A contact $3$-manifold $(Y,\xi)$ is said to be supported by an open book decomposition $(\Sigma,\psi,\Psi)$, if there exists a contact form $\alpha$ on $(Y,\xi)$ adapted to this open book decomposition.
\end{defn}
%\begin{rem} \label{rem:isotopy}
%If a contact $3$-manifold $(M,\xi)$ is supported by an open book decomposition $(\Sigma,\psi,\Psi)$, and $\psi': \Sigma \to \Sigma$ is a diffeomorphis which is the identity near $\partial \Sigma$ and that is isotopic to $\psi$ via an isotopy that fixes $\partial \Sigma $ pointwise,  then $(M,\xi)$ also supported by open book decomposition $(\Sigma,\psi',\Psi')$. This follows easily from the face that for such an $\psi'$ there is a diffeomorphism of $M(\psi)$ to $M(\psi')$ that takes pages the pages.
%\end{rem}

The following result of Giroux shows the central role played by open book decompositions in $3$-dimensional contact topology.

\begin{thm}[Giroux] \label{t:Giroux}
Given a contact $3$-manifold $(Y,\xi)$ there exists an open book decomposition of~$Y$ supporting~$(Y,\xi)$. Moreover, the open book decomposition can be chosen to have connected binding and pseudo-Anosov monodromy.

Two contact structures supported by diffeomorphic open book decompositions are diffeomorphic.
\end{thm}
\begin{rem}
The statement that the open book decomposition can be chosen to have pseudo-Anosov monodromy and connected binding is due to Colin and Honda \cite{CHstab}. 
\end{rem}

\subsection{Existence of transverse knots that force topological entropy}

In this section we present two different arguments to prove theorem \ref{theorem:existence}. We first remark that the binding $\mathcal B$ of an open book decomposition $(\Sigma,\psi,\Psi)$ that supports a contact structure $\xi$ is always a transverse link.

By Giroux's theorem \ref{t:Giroux}, it is clear that theorem \ref{theorem:existence} will follow once we establish the following result. 
\begin{thm} \label{t:obentropy}
Let $(Y,\xi)$ be a closed $3$-dimensional contact manifold that admits Reeb flows with vanishing topological entropy. Let  $(\Sigma,\psi,\Psi)$ be an open book decomposition that supports  $(Y,\xi)$ and satisfies
\begin{itemize}
    \item $\partial \Sigma$ is connected and the monodromy of the first return map $\psi$ is pseudo-Anosov map.
\end{itemize}

 Then there exists an open book decomposition $(\Sigma,\psi',\Psi')$ diffeomorphic to $(\Sigma,\psi,\Psi)$ that also supports $(Y,\xi)$ and whose binding $\mathcal{B}'$ forces topological entropy.
%$\lambda_0$ adapted to the open book $(\Sigma,\psi,\Psi)$ such that $CH_{\mathcal B}(\lambda_0)$ has exponential homotopical growth. It follows that $\mathcal B$ forces topological entropy.
\end{thm}
We now give two proofs of theorem \ref{t:obentropy}, thus presenting two ways to prove theorem \ref{theorem:existence}.

The first proof uses theorem \ref{theorem:growth-entropy} and the second proof uses theorem \ref{theorem:growthleg}.

%\begin{thm} \label{theorem:existence}
%Let $(M,\xi)$ be a $3$-dimensional contact manifold supported by an open book $(\Sigma,\psi,\Psi)$ whose binding $\mathcal B$ in connected and such that the $\psi$ map is isotopic to a pseudo-Anosov map. Then there exists a contact form $\lambda_0$ adapted to the open book $(\Sigma,\psi,\Psi)$, such that $CH_{\mathcal B}(\lambda_0)$ has exponential homotopical growth. It follows that $\mathcal B$ forces topological entropy.
%\end{thm}

\subsubsection{\it First proof of Theorem \ref{t:obentropy}} \label{sec:firstproof}

\

\textit{Step 1: The special contact form $\lambda_0$.}

Let $(Y,\xi)$ be a closed contact $3$-manifold as in the statement of the theorem. By theorem \ref{t:Giroux} we know that there exists an open book decomposition $(\Sigma,\psi,\Psi)$ supporting $(Y,\xi)$ with the properties that $\partial \Sigma$ is connected and the monodromy of the first return map $\psi$ is pseudo-Anosov map.

Consider an area form $\omega$ on $\Sigma$. We can assume that $\psi$ preserves $\omega$. The reason is that every diffeomorphism of $\Sigma$ that is the identity in a neighbourhood of $\partial \Sigma$ is isotopic to an area preserving one via an isotopy that fixes a neighbourhood of $\partial \Sigma$, and thus  the diffeomorphism class of any open book decomposition contains an open book decomposition whose monodromy map preserves $\omega$. It then follows from Giroux's theorem that, up to choosing a diffeomorphic open book decomposition, we can always assume that for the supporting open book decomposition $(\Sigma,\psi,\Psi)$ the map $\psi$ is area preserving. 

We now define a contact form $\sigma_\Sigma$ on the mapping torus $\Sigma(\psi)$. For this we first consider a collar neighourhood $N$ of $\partial \Sigma$ on which $\psi$ is the identity. We introduce coordinates $(r,\theta) \in [1,1+\delta] \times \partial \Sigma$ for $N$, where the boundary of $\partial \Sigma$ corresponds to $r=1$, and such that $\omega= -dr \wedge d \theta$ in these coordinates. Let $\beta$ be a primitive of $\omega$ which equals $(2-r) d\theta$ on $N$. We then choose $\delta>0$ such that the formula
\begin{equation}
dt + \delta\big( (1-\chi(t))\beta + \chi(t)\psi^*\beta\big)
\end{equation}
defines a contact form on $\Sigma(\psi)$,
where $t \in [0,2\pi]$ and 
$\chi:[0,2\pi] \to [0,1]$ is a monotone function such that $\chi(0)=0$, $\chi(2\pi)=1$ and $\chi'$ has support in $(0,2\pi)$. Let $$V:=[0,2\pi]\times N / \sim $$
where $(0,x) \sim (2\pi,x)$ for each $x\in N$.
We then perturb $dt + \delta((1-\chi(t))\beta + \chi(t)\psi^*\beta)$, by a perturbation with support inside $\Sigma(\psi) \setminus N$, to a contact form $\sigma_\Sigma$ such that
\begin{align} \label{eq:mptransv}
 \mbox{all Reeb orbits in } \Sigma(\psi) \setminus N \mbox{ are non-degenerate,} \\
 \nonumber \mbox{and } X_{\sigma_\Sigma} \mbox{ is transverse to } \Sigma \times \{t\} \mbox{ for all } t \in [0,2\pi].  
\end{align}

We proceed to construct  a contact form $\sigma_S$ on $\overline \D \times \partial \Sigma$.
Let $\overline{\D}$ be the closed unit disk in $\R^2$.
Consider polar coordinates $(\ovr, \ot) \in (0,1] \times S^1$ on  $\overline{\D}\setminus \{0\}$ and the coordinate $\otheta$ on $\partial \Sigma$.
We can then consider coordinates $(\ovr, \ot, \otheta) $ on $\overline{\D} \setminus \{0\} \times\partial \Sigma$.
We pick a smooth function $f \colon (0,1] \to \R$ such that 
\begin{itemize}
\item $f'<0$,
\item $f(\ovr) = 2-\ovr$ on a neighbourhood of $1$,
\item $f(\ovr) = 2 -\ovr^4$ on a neighbourhood of $0$.
\end{itemize}

\smallskip \noindent
We pick another smooth function $g \colon (0,1]\to \R$ satisfying
\begin{itemize}
\item $g'>0$ on $(0,1)$, 
\item $g(1) =1$ and all derivatives of $g$ vanish at $1$,
\item $g(\ovr) = \frac{\ovr^2}{2}$ on a neighbourhood of $0$.
\end{itemize}

Define the $1$-form 
\begin{equation} \label{e:la}
\sigma_{S} (\ovr, \ot, \otheta) \,=\, 
                 g(\ovr)  d\ot + \delta f(\ovr)  d\otheta 
\end{equation}
on $\overline{\D} \setminus \{0\} \times S^1$.
Then 
\begin{equation} \label{e:dla}
\sigma_{S} \wedge d \sigma_{S} \,=\, 
\delta  h(\ovr)  d\ovr \wedge d \ot \wedge d\otheta
\end{equation}
where $h(\ovr) = (fg'-f'g)(\ovr)$.
It follows that $\sigma_{S} $ is a contact form on $\overline{\D} \setminus \{0\}  \times \partial \Sigma$.
For $\ovr$ near $0$ we have $h(\ovr) = \ovr (2+\ovr^4)$,
whence $\sigma_{S} $ extends to a smooth contact form on 
$\overline{\D} \times \partial \Sigma$, that we also denote by $\sigma_{S} $.
The Reeb vector field of $\sigma_{S} $ is given by
\begin{equation} \label{e:Xle}
X_{\sigma_{S} } (\ovr, \ot, \otheta) \,=\, 
\frac{1}{h(\ovr)} \left( \frac 1 \delta  g'(\ovr)  \partial_{\otheta} - f'(\ovr)  \partial_\ot \right).
\end{equation}
It follows that $X_{\sigma_{S} }$ is tangent to the tori $\T_{\ovr} := \{ \ovr = \text{const} \}$,
and that for each $\ovr \in (0,1]$ the flow of $X_{\sigma_{S} }$ is linear:
\begin{equation} \label{e:linear}
\phi_{\sigma_{S} }^s (\ovr, \ot, \otheta) \,=\, \left( \ovr, \ot-\frac{f'(\ovr)}{h(\ovr)}  s, \otheta + \frac{g'(\ovr)}{\delta  h(\ovr)}  s \right) .
\end{equation}

In particular, using our choices of $f$ and $g$ we see that $X_{\sigma_{S}} = \partial_{\ot}$
on the boundary torus~$\T_1$,
and that $X_{\sigma_{S}} = \frac{1}{2 \delta} \partial_{\otheta}$ is tangent along the core circle $C= \{ \ovr=0 \}$ 
of the solid torus, and gives the positive orientation to $\partial \Sigma$.
Furthermore, \eqref{e:Xle} shows that $X_{\sigma_{S}}$ is positively transverse to the half-open annuli $A_{\ot}$.

We let $\sigma$ be the contact form on $Y(\psi)$ obtained by gluing $\sigma_\Sigma$ and $\sigma_S$ along their common boundary. It follows from the expressions for $\sigma_\Sigma$ in a neighbourhood of $\partial \Sigma(\psi)$ and for $\sigma_S$ in a neighbourhood of $\partial (\D \times \partial \Sigma)$, that the two contact forms indeed glue to give the contact form $\sigma$. The Reeb vector field of $\sigma$ is transverse to $\Sigma \times \{t\}$ and $A_t$ for every $t \in [0,2\pi]$: this follows from \eqref{eq:mptransv} and \eqref{e:linear}.

We now define the contact form $\widetilde \sigma$ on $Y$ by
\begin{equation}
    \widetilde \sigma := \Psi_* \sigma,
\end{equation}
where we recall that $\Psi:Y(\psi) \to Y$ is the diffeomorphism in the definition of the open book decomposition. Since the Reeb vector field $X_\sigma$ is transverse to $\Sigma \times \{t\}$ and $A_t$ for every $t \in [0,2\pi]$, we conclude that $\widetilde \sigma$ is adapted to the open book decomposition $(\Sigma,\psi,\Psi)$. 

It then follows that the contact manifold $(Y, \ker \widetilde \sigma)$ is supported by $(\Sigma,\psi,\Psi)$. Since $(Y,\xi)$ is also supported by $(\Sigma,\psi,\Psi)$ we conclude that there is a contactomorphism $\Upsilon:  (Y, \ker \widetilde \sigma) \to  (Y,\xi) $, and we define 
\begin{equation} \label{eq:specialform}
    \lambda_0 := \Upsilon^* \widetilde \sigma.
\end{equation}

\ 

{\it Step 2: Some properties of $CH_{\Upsilon(\mathcal{B})}(\lambda_0)$.}

In view of the contactomorphisms $\Psi$ and $\Upsilon$ we have the following isomorphisms of cylindrical contact homologies: $CH_{\Psi^{-1}(\mathcal{B})}(\sigma) \cong CH_{\mathcal{B}}(\widetilde \sigma) \cong CH_{\Upsilon(\mathcal{B})}(\lambda_0)$. It follows that the exponential growth rate of $CH_{\Upsilon(\mathcal{B})}(\lambda_0)$ is the same as the exponential growth rate of $CH_{\Psi^{-1}(\mathcal{B})}(\sigma)$.

We start quoting the following result from \cite{Alves-Cylindrical} about the contact form $\sigma_\Sigma$. In the proof of \cite[Proposition 11]{Alves-Cylindrical} it is shown that the set $\Delta^T(\sigma_\Sigma)$ of free homotopy classes of loops in $\Sigma(\psi)$ that 
\begin{itemize}
    \item are not homotopic to a curve completely contained in $\partial \Sigma(\psi)$,
    \item contain only non-degenerate Reeb orbits with action $\leq T$,
    \item contain an odd number of Reeb orbits,
\end{itemize}
satisfies
\begin{equation} \label{eq:growthhomclasses}
    \limsup_{T \to +\infty} \frac{\log \# \Delta^T(\sigma_\Sigma) }{T} > 0.
\end{equation}
This estimate uses the fact that the monodromy of $\psi$ is pseudo-Anosov. In particular the surface $\Sigma$ has to have positive genus, since it has negative Euler characteristic and only one boundary component.

We now remark that the manifold $\Sigma(\psi)$ inside $Y(\psi) \setminus \Psi^{-1}(\mathcal{B})$ is a deformation retract of $Y(\psi) \setminus \Psi^{-1}(\mathcal{B})$. It follows that two distinct free homotopy classes of loops in $\Sigma(\psi)$ remain distinct in $Y(\psi) \setminus \Psi^{-1}(\mathcal{B})$. Therefore we can view $\Delta^T(\sigma_\Sigma)$ as a set of distinct free homotopy classes in $Y(\psi) \setminus \Psi^{-1}(\mathcal{B})$.

\

{\it Step 3: Computation of the exponential growth rate of $CH_{\Upsilon(\mathcal{B})}(\lambda_0)$.}

In order to finish the proof we will show that for each element $\rho \in \Delta^T(\sigma_\Sigma)$ the triple $(\sigma,\Psi^{-1}(\mathcal{B}), \rho )$ satisfies the PLC condition of section \ref{section:growth-entropy} and that $CH^\rho_{\Psi^{-1}(\mathcal{B})}(\sigma) \neq 0$. 

We first remark that none of the Reeb orbits of $\sigma$ in $Y(\psi) \setminus \Psi^{-1}(\mathcal{B})$ is contractible in $Y(\psi) \setminus \Psi^{-1}(\mathcal{B})$. The reason for this is that all these Reeb orbits have positive intersection number with the surfaces $\Sigma\times \{t\} \cup A_t$, and since the boundary of these surfaces is $\Psi^{-1}(\mathcal{B})$, we conclude that they are all linked with $\Psi^{-1}(\mathcal{B})$.

Secondly we note that any disk in $Y(\psi)$ whose boundary is a cover of $\Psi^{-1}(\mathcal{B})$ must have an interior intersection point with $\Psi^{-1}(\mathcal{B})$ as shown in \cite[Lemma 6.6]{Momin}.

Let now $\rho \in \Delta^T(\sigma_\Sigma)$. We claim that all Reeb orbits of $\sigma$ in $\rho $ are in the interior of $\Sigma(\psi)$. To prove this we let $\gamma \in \rho$ be a Reeb orbit of $\sigma$, and assume that there is a cylinder that does not intersect $\Psi^{-1}(\mathcal{B})$ and whose boundaries are $\gamma$ and a Reeb orbit contained in $Y(\psi) \setminus (\Sigma(\psi) \cup \Psi^{-1}(\mathcal{B}))$. If we look at the image of this cylinder by the deformation retract that takes $Y(\psi) \setminus\Psi^{-1}(\mathcal{B})$ to $\Sigma(\psi)$, we obtain a cylinder contained in $\Sigma(\psi)$ with one boundary being $\gamma$ and the other being a curve in $\partial \Sigma(\psi)$. But the existence of this cylinder contradicts the fact that $\rho \in \Delta^T(\sigma_\Sigma)$.

We now claim that for $\rho \in \Delta^T(\sigma_\Sigma)$ there cannot be a cylinder $Cyl$ whose boundary is the union of a Reeb orbit $\gamma \in \rho$ and a multiple cover of $\Psi^{-1}(\mathcal{B})$. To prove this one first shows that if such a cylinder exists then there is a cylinder contained in $Y(\psi) \setminus\Psi^{-1}(\mathcal{B})$ whose boundary is the union of $\gamma$ with a curve completely contained in $Y(\psi) \setminus (\Sigma(\psi) \cup \Psi^{-1}(\mathcal{B}))$. This cylinder is obtained as the intersection of $Cyl$ with $Y(\psi) \setminus V$, where $V$ is a sufficiently small tubular neighbourhood of $\Psi^{-1}(\mathcal{B})$. We then reason as in the previous paragraph to obtain a contradiction.

This discussion implies that for all $ \rho \in \Delta^T(\sigma_\Sigma)$ we have that $(\sigma,\Psi^{-1}(\mathcal{B}), \rho )$ satisfies the PLC condition. Moreover, for each such $\rho$ all the Reeb orbits of $\sigma$ in the class $\rho$ are contained in $\Sigma(\psi)$. By the definition of $\Delta^T(\sigma_\Sigma)$ we then conclude that $\rho$ contains an odd number of Reeb orbits in $Y(\psi)$. For Euler characteristic reasons it follows that  
$CH^\rho_{\Psi^{-1}(\mathcal{B})}(\sigma)\neq 0$. 
We conclude that $\Delta^T(\sigma_\Sigma) \subset \Omega^T_{\Psi^{-1}(\mathcal{B})}(\sigma) $, where $\Omega^T_{\Psi^{-1}(\mathcal{B})}(\sigma) $ is the set used in the definition of the exponential homotopical growth rate in section \ref{section:homgrowthCH}. The exponential homotopical growth rate of $CH_{\Psi^{-1}(\mathcal{B})}(\sigma)$, and consequently of $CH_{\Upsilon(\mathcal{B})}(\lambda_0)$, then follows from \eqref{eq:growthhomclasses}. 

Lastly, we remark that $\Upsilon(\mathcal{B})$ is the binding of the open book decomposition $(\Sigma,\psi,\Upsilon\circ \Psi)$ which is diffeomorphic to $(\Sigma,\psi, \Psi)$ and supports $(Y,\xi)$.
Invoking theorem \ref{theorem:growth-entropy} we conclude that $\Upsilon(\mathcal{B})$ forces topological entropy.
\qed

\subsubsection{\it Second proof of Theorem \ref{t:obentropy}}

The structure of the second proof of Theorem \ref{theorem:existence} is very similar to the first, but this time we will use theorem \ref{theorem:growthleg}.

Let $(Y,\xi)$ be a closed contact $3$-manifold as in the statement of the theorem and $(\Sigma,\psi,\Psi)$ be an open book decomposition supporting $(Y,\xi)$, such that $\Sigma$ has only one boundary component and the monodromy of $\psi$ is pseudo-Anosov. Again we assume that $\psi$ is area preserving for an area form $\omega$ in $\Sigma$. Let $\beta$ be a primitive of $\omega$ which equals $(2-r) d\theta$ on $N$. We then choose $\delta>0$ such that the formula
\begin{equation}
dt + \delta\big( (1-\chi(t))\beta + \chi(t)\psi^*\beta\big)
\end{equation}
defines a contact form on $\Sigma(\psi)$. We denote the contact structure associated to this contact form by $\xi_\Sigma$.

We now follow \cite[Section 4.1]{A2} to obtain a pair of disjoint Legendrian knots $\Lambda$ and $\widehat{\Lambda}$ in $(\Sigma(\psi),\xi_\Sigma)$ which satisfy:
\begin{itemize}
    \item both $\Lambda$ and $\widehat{\Lambda}$ are contained in the interior of $[0,\pi] \times \Sigma \subset \Sigma(\psi) $ and are graphs over embedded curves in $\Sigma$ that are not contractible and are not homotopic to the boundary of $\Sigma$.
\end{itemize}

By a sufficiently small perturbation of $dt + \delta\big( (1-\chi(t))\beta + \chi(t)\psi^*\beta\big)$ supported in the interior of  $[0,2\pi] \times \Sigma \subset \Sigma(\psi) $ we can obtain a contact form $\sigma_\Sigma$ on $(\Sigma(\psi), \xi_\Sigma)$ satisfying
\begin{align} \label{eq:mptransv1}
 \mbox{all Reeb chords from } \Lambda \to \widehat{\Lambda} \mbox{ are transverse,} \\
 \nonumber \mbox{and } X_{\sigma_\Sigma} \mbox{ is transverse to } \Sigma \times \{t\} \mbox{ for all } t \in [0,2\pi].  
\end{align}
It follows from this that 
\begin{itemize}
    \item $\sigma_\Sigma$ is hypertight in the complement of $\Psi^{-1}(\mathcal{B}$,
    \item there are no Reeb chords of $\sigma_\Sigma$ from $\Lambda$ to itself that vanish in $\pi_1(\Sigma(\psi),\Lambda)$,
    \item there are no Reeb chords of $\sigma_\Sigma$ from $\widehat\Lambda$ to itself that vanish in $\pi_1(\Sigma(\psi),\widehat\Lambda)$,
\end{itemize}

We now quote the following result from \cite{A2} about the contact form $\sigma_\Sigma$. In the proof of \cite[Section 4.4]{A2} it is shown that if we denote by $\Delta^T(\sigma_\Sigma, \Lambda \to \widehat{\Lambda})$ the set of homotopy classes of paths in $\Sigma(\psi)$ from $\Lambda \to \widehat{\Lambda}$ that 
\begin{itemize}
    \item contain only non-degenerate Reeb chords with action $\leq T$,
    \item contain an odd number of Reeb chords,
\end{itemize}
then we have
\begin{equation} \label{eq:growthhomclasses1}
    \limsup_{T \to +\infty} \frac{\log \# \Delta^T(\sigma_\Sigma,\Lambda \to \widehat{\Lambda}) }{T} > 0.
\end{equation}
This estimate uses the fact that the monodromy of $\psi$ is pseudo-Anosov and that the Legendrian knots $\Lambda$ and $\widehat{\Lambda}$ are not homotopic to curves contained in $\partial \Sigma(\psi)$. Again $\Sigma$ has to have positive genus, since it has negative Euler characteristic and only one boundary component.

We then glue $\sigma_\Sigma$ to the contact form $\sigma_S$ as defined in \eqref{e:la} to obtain a contact form $\sigma$ on $Y(\psi)$.
The proof that $LCH_{\Psi^{-1}(\mathcal{B}}(\sigma,\Lambda \to \widehat \Lambda)$ grows exponentially is now an obvious adaptation of steps $2$ and $3$ of the first proof of theorem \ref{theorem:existence} that we gave in section \ref{sec:firstproof}.

Lastly, we remark that $\Upsilon(\mathcal{B})$ is the binding of the open book decomposition $(\Sigma,\psi,\Upsilon\circ \Psi)$ which is diffeomorphic to $(\Sigma,\psi, \Psi)$ and supports $(Y,\xi)$.
Invoking theorem \ref{theorem:growthleg} we conclude that $\Upsilon(\mathcal{B})$ forces topological entropy. 
\qed

\bibliographystyle{plain}
\bibliography{ReferencesS5}

\end{document}